\documentclass{amsart}
\usepackage{amssymb}
\usepackage{amsmath,amssymb,hyperref,mathrsfs,color}
\usepackage{paralist}
\usepackage{calc}
 \newtheorem{The}{Theorem}[section]
 \newtheorem{Cor}[The]{Corollary}
 \newtheorem{Lem}[The]{Lemma}
 \newtheorem{Pro}[The]{Proposition}
 \theoremstyle{definition}
 \newtheorem{defn}[The]{Definition}
 \theoremstyle{remark}
 \newtheorem{Rem}[The]{Remark}
 
 \numberwithin{equation}{section}

\begin{document}
\title[global viscosity solutions]{global viscosity solutions for eikonal equations on class A Lorentzian 2-tori}
\author{Liang Jin \and Xiaojun Cui}
\address{Department of Mathematics, Nanjing University,
Nanjing 210093, China}
\email{jinliangbruce@gmail.com}
\address{Department of Mathematics, Nanjing University,
Nanjing 210093, China}
\email{xcui@nju.edu.cn}
\date{\today}

\thanks{Both authors are supported by the National Natural Science Foundation of China (Grants 11271181, 11571166), the Project Funded by the Priority Academic Program Development of Jiangsu Higher Education Institutions (PAPD) and the Fundamental Research Funds for the Central Universities}
\maketitle

\begin{abstract}
  On the Abelian cover $(\mathbb{R}^{2},g)$ of a class A Lorentzian 2-torus $(\mathbb{T}^{2},g)$, we showed the existence of global viscosity solutions to the eikonal equation
  $$
  g(\nabla u,\nabla u)=-1
  $$ associated to those homologies in the interior of the homology cone. Some other related dynamical properties are also considered. As an application of the main results, we study the differentiability of the unit sphere of the stable time separation associated to the class A Lorentzian 2-torus.
\end{abstract}

\section{Introduction and a brief survey of preceding works}
Aubry-Mather theory for geodesic flows on class A Lorentzian 2-tori was established in E. Scheling's diploma thesis \cite{Sc}. More recently, higher dimensional generalizations to class A spacetimes were obtained by S. Suhr \cite{Su1},\cite{Su2}, where spacetime means a time-oriented $C^{\infty}$ Lorentzian manifold. Motivated by the relationship between Aubry-Mather theory and weak KAM theory in Tonelli Lagrangian systems, one would like to investigate an analogy of weak KAM theory in the setting of Lorentzian geodesic flow. However, due to the non-positive-definiteness of the Lorentzian metric, problems seem to be much more complicated. Fortunately, based on the topological properties deduced by the fact that the dimension of the configuration space we considered in this article is of 2, we could obtain the existence of global viscosity solutions to the eikonal equation on the Abelian cover associated to every homology in the interior of the homology cone. Moreover, the viscosity solutions we obtained in the present article present some properties of weak KAM type just as in the classical, namely positive-definite, case.

One might be interested in the definition and basic properties of class A spacetimes, since class A spacetimes are proved to be suitable settings for developing variational methods for  geodesic flows on Lorentzian manifolds. We strongly recommend \cite{Su1} as a considerable comprehensive reference on this topic, and our reformulation could be seen as a brief survey of this paper. Here, some notations must be clarified.

Let us consider a closed and connected spacetime $(M,g)$. For $p\in M$, it will be called an event from the viewpoint of general relativity, or be called a point from the viewpoint of mathematics. The wording we shall use will depend on the context. As usual, $\pi:\overline{M}\rightarrow M$ is the Abelian cover of $M$. We concentrate on two metrics on $M$: one is of course the Lorentzian metric $g$, the other is an auxiliary complete Riemannian metric $g_{R}$. We shall denote the lifts of $g$ and $g_{R}$ to $\overline{M}$ and all other objects associated with them by the same letter as on $M$ when there is no confusion. Let us define these objects one by one: the length functionals associated to $g$ and $g_{R}$ are denoted by $L^{g}(\cdot)$ and $L^{g_{R}}(\cdot)$ respectively; the distance functions associated to $g$ and $g_{R}$ are denoted by $d(\cdot,\cdot)$ and $d_{R}(\cdot,\cdot)$ respectively. For the Lorentzian manifold $(M,g)$ or $(\overline{M},g)$, we denote the causal future (past) at an event $p$ by $J^{+}(p)$ ($J^{-}(p)$); the chronological future (past) at event $p$ by $I^{+}(p)$ ($I^{-}(p)$). For the Riemannian manifold $(M,g_{R})$ or $(\overline{M},g_{R})$, the induced norm on $T_{p}M$ or $T_{p}\overline{M}$ is denoted by $|\cdot|_{p}$, and dist$_{p}(\cdot,\cdot)$ is the corresponding metric on tangent spaces; the stable norm associated to $(M,g_{R})$ (not $(\overline{M},g_{R})$ !) on $H_{1}(M,\mathbb{R})$ is denoted by $\|\cdot\|$, and dist$_{\|\cdot\|}(\cdot,\cdot)$ is the corresponding metric on $H_{1}(M,\mathbb{R})$.

\begin{defn}[\cite{B-E-E},\cite{Su2}]\label{class A}
Let $(M,g)$ be a connected spacetime, 
\begin{enumerate}
  \item $(M,g)$ is causal if there is no causal loops.
  \item $(M,g)$ is globally hyperbolic if it is causal and $J^{+}(p)\cap J^{-}(q)$ is compact (could be empty) for every pair of events $p, q\in M$.
  \item $(M,g)$ is vicious if $M=I^{+}(p)\cap I^{-}(p)$ for some event $p\in M$.
  \item $(M,g)$ is of class A if it is compact, vicious and the Abelian cover $(\overline{M},g)$ is globally hyperbolic.
\end{enumerate}
\end{defn}

For a general closed, vicious spacetime $(M,g)$, there exists a cone $\mathfrak{T}$ in $H_{1}(M,\mathbb{R})$, which is an approximation to the causal future of every $p\in \overline{M}$ in $(\overline{M},g)$. Class A spacetimes could be easily characterized from general closed, vicious spacetimes by using the topological properties of such a homology cone and its dual.

\begin{Pro}[{\cite[Propositions 8, 9; Theorem 11]{Su1}}]\label{cone}
Let $(M,g)$ be a closed and vicious spacetime.
\begin{itemize}
  \item There is a unique cone $\mathfrak{T}$ in $H_{1}(M,\mathbb{R})$ such that there exists a constant $D(g,g_{R})<\infty$ with $dist_{\|\cdot\|}(J^{+}(p)-p,\mathfrak{T})\leq D(g,g_{R})$ for all $p\in\overline{M}$, where $J^{+}(p)-p:=\{q-p|q\in J^{+}(p)\}\subseteq H_{1}(M,\mathbb{R})$. Such a cone is called the stable time cone and one could define its dual $\mathfrak{T}^{*}\subseteq H^{1}(M,\mathbb{R})$ by
      \begin{equation}
      \mathfrak{T}^{*}=\{\alpha\in H^{1}(M,\mathbb{R})|\alpha|_{\mathfrak{T}}\geq0\}.
      \end{equation}
  \item The following three statements are equivalent:
      \begin{enumerate}
      \item $(M,g)$ is of class A;
      \item $\mathfrak{T}$ is a compact cone with nonempty interior $\mathfrak{T}^{\circ}$;
      \item The interior $(\mathfrak{T}^{*})^{\circ}$ of $\mathfrak{T}^{*}$ is nonempty and for every $\alpha\in(\mathfrak{T}^{*})^{\circ}$, there is a smooth closed transversal 1-form $\omega$ with $[\omega]=\alpha$ such that $ker\omega_{p}$ is spacelike in $(TM_{p},g_{p})$ for all $p\in M$.
\end{enumerate}
\end{itemize}
\end{Pro}

By Proposition \ref{cone}, the class A condition for spacetime $(M,g)$ implies the existence of a closed transversal 1-form for the cone structure of future-directed vectors in $(M,g)$. This leads to an easy corollary that will be used frequently later.

\begin{Cor}[{\cite[Corollary 12]{Su1}}]\label{causal}
Let $(M,g)$ be a class A spacetime. Then there exists a constant $C(g,g_{R})<\infty$ such that $L^{g_{R}}(\gamma)\leq C(g,g_{R})d_{R}(p,q)$ for all $p,q\in\overline{M}$ and all causal curves $\gamma$ connecting $p$ with $q$.
\end{Cor}

Now we restrict ourselves to consider a two dimensional connected, oriented, closed spacetime $(M^{2},g)$. It is well known that in this case $M^{2}$ is diffeomorphic to $\mathbb{T}^{2}:=\mathbb{R}^{2}/\mathbb{Z}^{2}$. Replacing $M$ and $\overline{M}$ by $\mathbb{T}^{2}$ and $\mathbb{R}^{2}$ respectively, we continue to use the notations defined in the third paragraph. By Definition \ref{class A}, $(\mathbb{T}^{2},g)$ is of class A if and only if it is vicious and its Abelian cover $(\mathbb{R}^{2},g)$ is globally hyperbolic. From now on, every causal curve under consideration is future-directed, for the definition of future-directed, see \cite[Chapter 3, Page 54]{B-E-E}.

The following concept named asymptotic direction is convenient for our statements in the setting of class A 2-torus, it will appear in almost all theorems in this article. Before giving the definition, we would like to clarify a related concept. We say an arbitrary causal curve $\gamma:I\rightarrow\mathbb{R}^{2}$ is unbounded in both directions if $I$ is open, say $I=(a,b)$, and for every $t\in I$,
\begin{equation}\label{eq:ub}
\begin{split}
\lim_{s\searrow a}d_{R}(\gamma(t),\gamma(s))=\infty,\\
\lim_{s\nearrow b}d_{R}(\gamma(t),\gamma(s))=\infty.
\end{split}
\end{equation}
Similarly we can define unboundedness in the past (future) direction by requiring that $I$ is one-sided open and the first (second) equality of Equation \ref{eq:ub} is satisfied. One may notice that the concept of unboundedness coincides with the concept of partially imprisoned, see \cite[Page 62]{B-E-E}.

It is obvious that for every future inextendible causal curve $\gamma:I\rightarrow(\mathbb{R}^{2},g)$, $I=[a,b)$ or $(a,b)$ and for any $t\in I$, $L^{g_{R}}(\gamma|_{[t,b)})=\infty$ since otherwise it will have a future endpoint. Then Corollary \ref{causal} implies that
\begin{equation}
\lim_{s\nearrow b}d_{R}(\gamma(t),\gamma(s))=\infty.
\end{equation}
So every future inextendible causal curve is unbounded in the future direction. One could easily formulate and prove that every inextendible causal curve is unbounded in both directions.

Thus if we parametrize a future inextendible causal curve with a past endpoint (resp. an inextendible causal curve) by the $g_{R}$-arc length, the domain should be $\mathbb{R}_{+}:=[0,\infty)$ (resp. $\mathbb{R}$). In the remaining context of this section, causal curves are always parametrized by $g_{R}$-arc length.

In the following definition, a curve $\gamma:I\rightarrow\mathbb{T}^{2}, I=\mathbb{R}_{+}:=[0,\infty)$ or $\mathbb{R}$ is unbounded in the future direction if any of its lifts satisfies the corresponding conditions above.
\begin{defn}\label{asym}
Denote the unique half line in a vector space that contains a given non-zero vector $\alpha$ by $\overline{\alpha}$.
\begin{enumerate}
  \item Let $\gamma:I\rightarrow\mathbb{T}^{2}$ be a causal curve unbounded in the future direction. If there exists a half line $l\subseteq H_{1}(\mathbb{T}^{2},\mathbb{R})$ emanating from the null-homology such that $dist_{\|.\|}(\tilde{\gamma}(T_{2})-\tilde{\gamma}(T_{1}),l)$ has a uniform upper bound independent of $[T_{1},T_{2}]\subseteq I$, where $\tilde{\gamma}$ is a lift of $\gamma$ to $\mathbb{R}^{2}$ and $\tilde{\gamma}(T_{2})-\tilde{\gamma}(T_{1})\in H_{1}(\mathbb{T}^{2},\mathbb{R})$, then we say $\gamma$ has the same asymptotic direction as the half line $l$. Since all half lines in $H_{1}(\mathbb{T}^{2},\mathbb{R})$ emanating from the null-homology form the spherization $SH_{1}(\mathbb{T}^{2},\mathbb{R})$, which is isomorphic to $S^{1}:=\{h|h\in H_{1}(\mathbb{T}^{2},\mathbb{R}),\|h\|=1\}$ in the sense of topology, we shall call the unique vector $\alpha\in S^{1}$ satisfying $l=\overline{\alpha}$ the asymptotic direction of $\gamma$. One easily see that this definition is independent of the choice of the lift $\tilde{\gamma}$.

  \item If the set $\overline{\alpha}\cap H_{1}(\mathbb{T}^{2},\mathbb{Z})_{\mathbb{R}}$ is nonempty, where $H_{1}(\mathbb{T}^{2},\mathbb{Z})_{\mathbb{R}}$ is the image of $H_{1}(\mathbb{T}^{2},\mathbb{Z})$ in $H_{1}(\mathbb{T}^{2},\mathbb{R})$, we call $\alpha\in S^{1}$ a rational asymptotic direction. Otherwise, we call $\alpha$ an irrational asymptotic direction.
  \item Define the set of causal asymptotic directions to be the set of asymptotic directions $\alpha\in S^{1}$ such that there exists a future inextendible causal curve $\gamma:I\rightarrow(\mathbb{T}^{2},g)$ that has the asymptotic direction $\alpha$.
  \item We shall say that  a future inextendible causal curve $\gamma:I\rightarrow(\mathbb{R}^{2},g)$ has asymptotic direction $\alpha$ if $\pi \circ \gamma$, the projection of $\gamma$ onto $\mathbb{T}^{2}$, has asymptotic direction $\alpha$.
\end{enumerate}
\end{defn}

\begin{Rem}
We equip $S^{1}$ with the induced metric topology from $(H_{1}(\mathbb{T}^{2},\mathbb{R}),\|\cdot\|)$.
\end{Rem}

By the time orientability of $(\mathbb{T}^{2},g)$, the Lorentzian metric $g$ uniquely defines two smooth future-directed lightlike vector fields $X_{1},X_{2}$ on $\mathbb{T}^{2}$ with $|X_{1,2}|_{p}=1$ for any  $p\in \mathbb{T}^2$.  By the foliation theory \cite[Part A, Section 4.3]{H-H}, there is a straight line (not half line!) $M^{i}\subseteq H_{1}(\mathbb{T}^{2},\mathbb{R})$ passing through the origin such that for any integral curve  $\gamma_{i}$ of $X_{i}$ and for any $s\leq t$,
$$
dist_{\|\cdot\|}(\tilde{\gamma_{i}}(t)-\tilde{\gamma_{i}}(s), M^{i})\leq D(g,g_{R}),
$$
where $\tilde{\gamma_i}$ is a lift of $\gamma_i$, $D(g,g_{R})$ is a constant depending only on $g$ and $g_{R}$ and $i=1,2$.

The class A condition for $(\mathbb{T}^{2},g)$ is equivalent to $M^{1}\neq M^{2}$ \cite[Satz 4.1]{Sc}. In this case, the integral curves of $X_{i} (i=1,2)$ do have an asymptotic direction defined before \cite[Section 4.1]{Su2}. We denote the asymptotic directions of the integral curves of $X_{i}$ and corresponding half lines by $m^{i}$ and $\overline{m}^{i}$ respectively. Using the concept of asymptotic direction, one could give another beautiful definition of stable time cone, which makes this concept more concrete in two the dimensional case, see \cite[Section 3.1, Page 6]{Su3}.

\begin{Pro}[{\cite[Section 3.1]{Su2}}]\label{cone'}
Let $(\mathbb{T}^{2},g)$ be a class A Lorentzian 2-torus. Then the convex hull of $\overline{m}^{1}\cup\overline{m}^{2}$ is the stable time cone $\mathfrak{T}$.
\end{Pro}

\begin{Rem}
Since there are  Lorentz metrics $g$ on 2-torus such that the integral curves of $X_{i}$ do not have an asymptotic direction (see \cite[Section 8.3]{Sc}), so we could not replace the straight lines $M_{i}$ in the above by half lines. Moreover, such phenomenon is stable under the perturbation of $g$. Of course, these metrics are not of class A. This fact also leads to the conclusion that class A metrics are not dense in the set of all smooth time orientable Lorentz metric on 2-torus. Interested readers could also refer to \cite{Su2}, which contains a description of lightlike foliations for non-class A metrics on $\mathbb{T}^{2}$.
\end{Rem}

It is now clear that the set of causal directions is just the set $S^{1}\cap\mathfrak{T}$. We identify $H_{1}(\mathbb{T}^{2},\mathbb{R})$ with $\mathbb{R}^{2}$ and give an orientation on it as usual, i.e.  the counter-clockwise is positive orientation. This orientation leads to a natural order on $S^{1}\cap\mathfrak{T}$.

\begin{defn}\label{ord}
Let $\alpha,\beta\in S^{1}\cap\mathfrak{T}$ be two distinct causal asymptotic directions. We define $\alpha<\beta$ if and only if the order pair $\{\alpha,\beta\}$ is positively oriented. The above order is well defined since $\mathfrak{T}$ cannot contain any one dimensional subspace of $\mathbb{R}^{2}$. We denote $(S^{1}\cap\mathfrak{T},<)$ by $[m^{-},m^{+}]$, where $m^{\pm}$ is defined by $\{m^{1,2}\}=\{m^{\pm}\}$ and $m^{-}<m^{+}$. Notice that no matter in sense of topology structure or order structure, $S^{1}\cap\mathfrak{T}$ is isomorphic to a closed interval, so the symbol we choose makes no confusion. We also define $(m^{-},m^{+}):=[m^{-},m^{+}]\setminus\{m^{\pm}\}$, we say the asymptotic directions belong to $(m^{-},m^{+})$ are timelike.
\end{defn}

To introduce the results of E. Scheling, we need one more definition, which is just an analogy of minimal geodesics in Riemannian case.

\begin{defn}\label{maximizer, ray, line}
Let $I$ be an arbitrary interval. We call a timelike (causal) curve $\gamma:I\rightarrow(\mathbb{R}^{2},g)$ a timelike (causal) maximizer if it satisfies
\begin{equation}
L^{g}(\tilde{\gamma}|_{[a,b]})=d(\tilde{\gamma}(a),\tilde{\gamma}(b))
\end{equation}
for all $[a,b]\subseteq I$, here $\tilde{\gamma}$ is a lift of $\gamma$.

We call a timelike (causal) curve $\gamma:\mathbb{R}_{+}\rightarrow(\mathbb{R}^{2},g)$ with a past endpoint a timelike (causal) ray if it is future inextendible and is a timelike (causal) maximizer; similarly, we call a timelike (causal) curve $\gamma:\mathbb{R}\rightarrow(\mathbb{R}^{2},g)$ a timelike (causal) line if it is inextendible and is a timelike (causal) maximizer.

We also call a timelike (causal) curve on $(\mathbb{T}^{2},g)$ a timelike (causal) maximizer (ray, line) if any of its lift is a timelike (causal) maximizer (ray, line).
\end{defn}

In his diploma thesis \cite{Sc}, E. Scheling proved the following results under the setting of geodesic flows on Lorentzian class A 2-tori, as an analogy of the results obtained by V. Bangert under the setting of monotone twist maps on annulus and geodesic flows on Riemannian 2-tori in his celebrated paper \cite{Ba 1}. To simplfy the notation, we identify  $H_{1}(\mathbb{T}^{2},\mathbb{Z})_{\mathbb{R}}$ with $H_{1}(\mathbb{T}^{2},\mathbb{Z})$ in the following context.

\begin{Pro}[\cite{Sc},\cite{Su2}]\label{causal maximizer}
Let $(\mathbb{T}^{2},g)$ be a class A Lorentzian 2-torus. Then
\begin{itemize}
  \item For every $h\in H_{1}(\mathbb{T}^{2},\mathbb{Z})\cap\mathfrak{T}$, there exists a closed causal line $\gamma:\mathbb{R}\rightarrow(\mathbb{T}^{2},g)$ with homology class $h$.
  \item Let $\gamma:\mathbb{R}\rightarrow(\mathbb{T}^{2},g)$ be a closed timelike line with minimal period $T$. Set $h=[\gamma|_{[0,T]}]\in H_{1}(\mathbb{T}^{2},\mathbb{Z})$. Then the class $h$ is irreducible in $H_{1}(\mathbb{T}^{2},\mathbb{Z})$, i.e. for any $h'\in H_{1}(\mathbb{T}^{2},\mathbb{Z})$ and $\lambda>0$ with $h=\lambda h'$, we have $\lambda=1$ and $h'=h$.
  \item For any asymptotic direction $\alpha\in[m^{-},m^{+}]$, there is a causal line $\gamma:\mathbb{R}\rightarrow(\mathbb{T}^{2},g)$ with asymptotic direction $\alpha$. In addition, for all $T_{1}\leq T_{2}$, $\tilde{\gamma}(T_{2})-\tilde{\gamma}(T_{1})$ lies at bounded distance from $\overline{\alpha}$. This distance only depends on $g$ and $g_{R}$.
\end{itemize}
\end{Pro}

We need several definitions to proceed. Let $\gamma$ be an inextendible causal line on $(\mathbb{R}^{2},g)$, by Proposition \ref{causal} and elementary topological knowledge, we know that $\mathbb{R}^{2}\setminus Im(\gamma)=U^{-}\cup U^{+}$, where $U^{\pm}$ are two connected components of $\mathbb{R}^{2}\setminus Im(\gamma)$ and are diffeomorphic to $\mathbb{R}^{2}$. One of these two connected components, say $U^{+}$ ($U^{-}$) satisfies the following condition: for every spacelike smooth curve $\xi$ that initiates from $p=\gamma(t_{0})$ such that $\xi(t),t>0$ is contained in $U^{+}$ ($U^{-}$), $\{\dot{\gamma}(t_{0}),\dot{\xi}(0)\}$ is positively (negatively) oriented. Using these obserations, we can define a relation between a point in $\mathbb{R}^{2}\setminus Im(\gamma)$ and the causal line $\gamma$.

\begin{defn}\label{<,>}
Let $(\mathbb{R}^{2},g)$ be the Abelian cover of a class A Lorentzian 2-torus and $\gamma$ a causal line on it.  A point $p\in\mathbb{R}^{2}\setminus Im(\gamma)$ satisfies $p>\gamma$ ($p<\gamma$) if and only if $p\in U^{+}$ ($p\in U^{-}$); a point $p\in\mathbb{R}^{2}$ satisfies $p\geq\gamma$ ($p\leq\gamma$) if and only if $p\in Im(\gamma)\cup U^{+}$ ($p\in Im(\gamma)\cup U^{-}$). A causal curve $\zeta>\gamma$ ($\zeta<\gamma$) if and only if $Im(\zeta)\subseteq U^{+}$ ($Im(\zeta)\subseteq U^{-}$); a causal curve $\zeta\geq\gamma$ ($\zeta\leq\gamma$) if and only if $Im(\zeta)\subseteq Im(\gamma)\cup U^{+}$ ($Im(\zeta)\subseteq Im(\gamma)\cup  U^{-}$).
\end{defn}

\begin{defn}\label{asy}
Let $\gamma_{1},\gamma_{2}:I\rightarrow(\mathbb{R}^{2},g)$ be two future inextendible causal curves, here $I=\mathbb{R}_{+}\hspace{0.1cm}\text{or}\hspace{0.1cm}\mathbb{R}$. We say $\gamma_{1}$ and $\gamma_{2}$ are asymptotic in the future direction if
$$
\lim_{t\rightarrow\infty}\max\{d_{R}(\gamma_{1}(t),\gamma_{2}(I)),d_{R}(\gamma_{2}(t),\gamma_{1}(I))\}=0.
$$
One easily formulate the definition of two inextendible causal curves that are asymptotic in the past direction.
\end{defn}

Like before, $\pi:\mathbb{R}^{2}\rightarrow\mathbb{T}^{2}$ is the Abelian cover over $\mathbb{T}^{2}$. The Deck transformations ($\cong\mathbb{Z}^{2}$) associated to $\pi$ act on $\mathbb{R}^{2}$ and are defined as the following:
\begin{equation*}
T:\mathbb{R}^{2}\times\mathbb{Z}^{2}\rightarrow\mathbb{R}^{2},
\hspace{0.2cm}(p,(j,k))\mapsto p+(j,k),\hspace{0.2cm}p\in\mathbb{R}^{2}.
\end{equation*}
$T_{(j,k)}:=T|_{(j,k)\times\mathbb{R}^{2}}$ is the translation by $(j,k)\in\mathbb{Z}^{2}$. By definition, $T_{(j,k)}:\mathbb{R}^{2}\rightarrow\mathbb{R}^{2}$ is an isometry w.r.t both $g$ and $g_{R}$ for every $(j,k)\in\mathbb{Z}^{2}$.

\begin{defn}
Let $(\mathbb{R}^{2},g)$ be the Abelian cover of a class A Lorentzian 2-torus and $\alpha$ be an asymptotic direction in $(m^{-},m^{+})$. We denote the set of all timelike lines on $(\mathbb{R}^{2},g)$ with asymptotic direction $\alpha$ by $\mathscr{M}_{\alpha}$.

If $\alpha$ is irrational, by the theorem follows and the periodicity of the Lorentzian metric, $\mathscr{M}_{\alpha}$ is totally ordered by $<$ defined in Definition \ref{<,>} and is invariant under the actions of $T_{(j,k)}$. Restricting the actions of $T_{j,k}$ on $\mathscr{M}_{\alpha}$, there is a unique minimal set, $\mathscr{M}_{\alpha}^{rec}$, in the sense of topological dynamics. We say timelike lines in $\mathscr{M}_{\alpha}^{rec}$ and their projections onto $\mathbb{T}^{2}$ are recurrent.

If $\alpha$ is a rational asymptotic direction, and $(q,p)\in\overline{\alpha}$ ($p,q$ relatively prime), define
$$
\mathscr{M}_{\alpha}^{+}:=\{\gamma\in\mathscr{M}_{\alpha}|T_{(q,p)}\gamma<\gamma\},\hspace{0.2cm}
\mathscr{M}_{\alpha}^{-}:=\{\gamma\in\mathscr{M}_{\alpha}|T_{(q,p)}\gamma>\gamma\}
$$
and
$$
\mathscr{M}_{\alpha}^{per}:=\{\gamma\in\mathscr{M}_{\alpha}|T_{(q,p)}\gamma=\gamma\}.
$$
In this case, the last class of timelike lines are called periodic, and we  also call they are recurrent.
\end{defn}

\begin{defn}\label{nb}
In the sequel, we call two causal lines $\overline{\gamma}, \underline{\gamma}:\mathbb{R}\rightarrow(\mathbb{R}^{2},g)$ with the same asymptotic direction are neighboring if the unique strip bounded by $\underline{\gamma}$ and $\overline{\gamma}$ does not contain any other recurrent (including periodic cases) causal lines.
\end{defn}

\begin{The}[\cite{Sc},\cite{Su2}]\label{structure for lines}
Let $(\mathbb{R}^{2},g)$ be the Abelian cover of a class A Lorentzian 2-torus, then:
\begin{enumerate}
  \item There exist causal lines for every asymptotic direction $\alpha\in[m^{-},m^{+}]$ and every causal line has an asymptotic direction $\alpha\in[m^{-},m^{+}]$. In addition, if $\gamma:\mathbb{R}\rightarrow(\mathbb{R}^{2},g)$ is a causal line with asymptotic direction $\alpha$, then for all $T_{1}\leq T_{2}$, $\gamma(T_{2})-\gamma(T_{1})$ lies at bounded distance from $\overline{\alpha}$. This distance depends only on $g$ and $g_{R}$.
  \item Let $\alpha\in(m^{-},m^{+})$ be an irrational asymptotic direction. Then any two distinct timelike lines with the same asymptotic direction $\alpha$ are disjoint. The set of points in $\mathbb{R}^{2}$ lying on a recurrent timelike line with asymptotic direction $\alpha$ either constitute a foliation on $\mathbb{R}^{2}$ or intersect every transversal in a Cantor set.  Moreover, there is a bijective mapping between pairs of neighboring timelikes lines $(\underline{\gamma}, \overline{\gamma})$ with asymptotic direction $\alpha$ and gaps of the Cantor set mentioned above, and any pair of such neighboring timelike lines $\underline{\gamma}, \overline{\gamma}:\mathbb{R}\rightarrow(\mathbb{R}^{2},g)$ are asymptotic in both directions.
  \item Let $\alpha\in(m^{-},m^{+})$ be a rational asymptotic direction. Then $\mathscr{M}_{\alpha}=\mathscr{M}_{\alpha}^{per}\cup\mathscr{M}_{\alpha}^{+}\cup\mathscr{M}_{\alpha}^{-}$.
      For $\gamma\in\mathscr{M}_{\alpha}^{+}\cup\mathscr{M}_{\alpha}^{-}$, there exist two neighboring periodic timelike lines $\underline{\gamma}, \overline{\gamma}$ such that $\underline{\gamma}<\gamma<\overline{\gamma}$. Each timelike line $\gamma\in\mathscr{M}_{\alpha}^{+}$ ($\gamma\in\mathscr{M}_{\alpha}^{-}$) is asymptotic to $\overline{\gamma}$ in the future (past) direction and asymptotic to $\underline{\gamma}$ in the past (future) direction. Besides these, in every strip bounded by two neighboring periodic timelike lines $\underline{\gamma}, \overline{\gamma}$, there exist at least two timelike lines: one in $\mathscr{M}_{\alpha}^{+}$ and the other in $\mathscr{M}_{\alpha}^{-}$.
  \item The asymptotic direction is continuous w.r.t. the $C^{0}$-topology on the space of causal lines, i.e. if $\gamma_{k}$ is a series of causal lines with asymptotic directions $\alpha_{k}$ and $\gamma_{k}$ converges to $\gamma$ in the $C^{0}$-topology with asymptotic direction $\alpha$, then $\alpha_{k}\rightarrow\alpha$ w.r.t. the topology defined on $[m^{-},m^{+}]$. Moreover, for any timelike asymptotic directions $\alpha_{k}\rightarrow\alpha$ and any $\gamma\in\mathscr{M}_{\alpha}^{rec}(\mathscr{M}_{\alpha}^{per})$, $\gamma$ is the $C^{0}$-limit of a sequence $\gamma_{k}\in\mathscr{M}_{\alpha_{k}}^{rec}(\mathscr{M}_{\alpha}^{per})$.
\end{enumerate}
\end{The}

This article is organized as follows. In Section 2, we introduce the definition of  viscosity solutions to the Lorentzian eikonal equation and state our main results, namely Theorems \ref{main theorem 1}, \ref{main theorem 2} and \ref{non-diff sts}. Section 3 is devoted to a proof of geodesical completeness of timelike rays (lines) with asymptotic directions in $(m^{-},m^{+})$. The intersection properties of timelike rays with asymptotic directions in $(m^{-},m^{+})$ is discussed in detail in Section 4. We show some priori regularities of the  Lorentzian Busemann functions that are necessary for the proof of our main results in Section 5. We give the proof of Theorems \ref{main theorem 1}, \ref{main theorem 2} in Section 6. As an application of Theorems \ref{main theorem 1}, \ref{main theorem 2}, we prove Theorem \ref{non-diff sts} in Section 7. The last section is given as an appendix, where we introduce some basic tools and results that are necessary in our proof of the main results.

\section{Statement of the main results}
In this section, we shall state the main results of this article. We need to introduce two preliminary notions: viscosity solution of eikonal equation and semiconcavity.

\begin{defn}\label{eq:eikonal}
Let $(M,g)$ be a Lorentzian manifold, and $\nabla$ be the gradient induced by the Lorentzian metric $g$. Associated to the Lorentzian metric $g$, there is a canonical Hamilton-Jacobi type equation on $(\overline{M},g)$, i.e. the eikonal equation
\begin{equation*}
g(\nabla u, \nabla u)=-1    \tag{$*$}
\end{equation*}
lies at the center of the study of global Lorentzian geometry.
\end{defn}

Let us interpret what is the viscosity solution of the eikonal equation ($*$). First we introduce the definition of generalized gradients.

\begin{defn}[{\cite[Definition 3.1.6, Proposition 3.1.7]{C-S}}]\label{generalized gradients}
If $u:\overline{M}\rightarrow\mathbb{R}$ is a continuous function defined on Lorentzian manifold $(\overline{M},g)$, then a vector $V\in T_{q}\overline{M}$ is called to be a subgradient (resp. supergradient) of $u$ at $q\in\overline{M}$, if there exists a neighborhood $O$ of $q$ and a $C^{1}$ function $\phi:O\rightarrow\mathbb{R}$ satisfies that $\phi(q)=u(q)$, $\phi(p)\leq u(p)$ (resp. $\phi(p)\geq u(p)$) for every $p\in O$ and $\nabla \phi(q)=V$.
\end{defn}
We denote by $\nabla^{-}u(q)$ (resp. $\nabla^{+}u(q)$) the set of subgradients (resp. supergradients) of $u$ at $q$.

\begin{defn}\label{vis}
A continuous function is called a viscosity subsolution of equation ($*$) if for any $q \in\overline{M}$,
$$g(V,V) \leq -1 \text{ for every } V \in \nabla^+u(q).$$

Similarly, a continuous function is called a viscosity supersolution of equation ($*$) if for any $q \in\overline{M}$,
$$g(V,V) \geq -1 \text{ for every } V \in \nabla^-u(q).$$

A continuous function is a viscosity solution if it is  a viscosity subsolution and a viscosity supersolution simultaneously.
\end{defn}

The second notion we need is semiconcavity.

\begin{defn}[{\cite[Definition 10.10]{V}}]\label{semiconcave}
Fix an auxiliary Riemannian metric $g_{R}$ on a smooth manifold $\overline{M}$ and let $U$ be an open subset of $\overline{M}$, a function $u:U\rightarrow\mathbb{R}$ is said to be semiconcave if there exists a $K>0$ such that for any constant-speed $g_{R}$-geodesic path $\gamma(t), t\in[0,1]$, whose image is included in $U$,
\begin{equation}\label{ieq:semiconcave}
(1-t)u(\gamma(0))+tu(\gamma(1))-u(\gamma(t))\leq K\frac{t(1-t)}{2}d^{2}_{R}(\gamma(0),\gamma(1)).
\end{equation}

A function $u:\overline{M}\rightarrow \mathbb{R}$ is said to be locally semiconcave if for each $q\in\overline{M}$ there is a neighborhood $U$ of $q$ in $\overline{M}$ such that \ref{ieq:semiconcave} holds true as soon as $\gamma(0),\gamma(1)\in U$.
\end{defn}

For $\alpha\in(m^{-},m^{+})$, let $\mathscr{R}_{\alpha}$ denote the set of timelike rays on $(\mathbb{R}^{2},g)$ with asymptotic direction $\alpha$. Based on all the definitions and results surveyed in Section 1, we could prove the following two theorems. In these two theorems, timelike rays in $\mathscr{R}_{\alpha}$ are parametrized as $g$-geodesics with domain $\mathbb{R}_{+}$.

\begin{The}\label{main theorem 1}
Let $(\mathbb{R}^{2},g)$ be the Abelian cover of a class A Lorentzian 2-torus. Then for every asymptotic direction $\alpha\in(m^{-},m^{+})$, there exist at least one weakly closed form $\omega_{\alpha}$ on $\mathbb{T}^{2}$ and a corresponding continuous function $b_{\alpha}:\mathbb{R}^{2}\rightarrow\mathbb{R}$ such that:
\begin{enumerate}
  \item $db_{\alpha}=(\pi^{-1})^{\ast}\omega_{\alpha}$ and $b_{\alpha}$ is a viscosity solution of the Lorentzian eikonal equation ($*$).

  \item $b_{\alpha}$ is Lipschitz and locally semi-concave, so its first and second derivatives exist almost everywhere.

  \item $b_{\alpha}$ is differentiable at  $q\in \mathbb{R}^{2}$ if and only if there exists a unique timelike ray $\gamma_{q}:\mathbb{R}_{+}\rightarrow \mathbb{R}^{2}$ in $\mathscr{R}_{\alpha}$, that satisfies
      \begin{equation}\label{eq:weakam1}
      b_{\alpha}(\gamma_{q}(t))-b_{\alpha}(\gamma_{q}(0))=t;
      \hspace{0.2cm}\gamma_{q}(0)=q
      \end{equation}
      and
      \begin{equation}\label{eq:weakam2}
      \dot{\gamma_{q}}(0)= -\nabla b_{\alpha}(q).
      \end{equation}
      For fixed $\alpha$ and $q\in\mathbb{R}^{2}$, $b_{\alpha}$ is differentiable at any $\gamma_{q}(t)$ for $t\in(0,\infty)$ and $\dot{\gamma_{q}}(t)=-\nabla b_{\alpha}(\gamma_{q}(t))$.

  \item Denote the set of all lifted timelike rays $\zeta:\mathbb{R}_{+}\rightarrow(\mathbb{R}^{2},g)$ that satisfy
      \begin{equation}
      b_{\alpha}(\zeta(t))-b_{\alpha}(\zeta(0))=t\hspace{0.2cm}\text{for any} \hspace{0.2cm}t\in\mathbb{R}_{+}
      \end{equation}
      by $\mathfrak{C}_{\alpha}$.
      Then $\mathfrak{C}_{\alpha}\subseteq\mathscr{R}_{\alpha}$.
\end{enumerate}
\end{The}

\begin{The}\label{main theorem 2}
Let $(\mathbb{R}^{2},g)$ be the Abelian cover of a class A Lorentzian 2-torus.
\begin{enumerate}
  \item For every irrational $\alpha\in(m^{-},m^{+})$, the function $b_{\alpha}$ satisfying all conditions listing in Theorem \ref{main theorem 1} is unique up to a constant. In addition, the set equation $\mathfrak{C}_{\alpha}=\mathscr{R}_{\alpha}$ always holds in this case.
  \item For every rational $\alpha\in(m^{-},m^{+})$, if there does not exist a foliation of $\mathbb{T}^{2}$ consisting of all periodic lines with asymptotic direction $\alpha$ as its leaves, then there are at least two different functions, say $b_{\alpha}^{+},b_{\alpha}^{-}$, satisfying all conditions listed in Theorem \ref{main theorem 1}. Moreover, the corresponding weakly closed forms represent different cohomology classes, i.e.
      \begin{equation}\label{ieq:cohomology class}
      [\omega_{\alpha}^{+}]\neq[\omega_{\alpha}^{-}].
      \end{equation}
\end{enumerate}
\end{The}

\begin{Rem}
 Since timelike maximizers could be reparametrized as geodesics of $(\mathbb{R}^{2},g)$, the main results imply that, as $g$-geodesics, every timelike ray and line with asymptotic direction $\alpha\in(m^{-},m^{+})$ is future geodesically complete (geodesically complete) respectively. We will prove this point in Section 3.
\end{Rem}

As an application of the main results, we could prove the following theorem concerning the differentiability of the unit sphere of the stable time separation $\mathfrak{l}$. For the definition of stable time separation, see Theorem \ref{sts}.

\begin{defn}
Since for every $\alpha\in(m^{-},m^{+})$, $\mathfrak{l}^{-1}(1)\cap\overline{\alpha}$ is a singleton, we define that the unit sphere $\mathfrak{l}^{-1}(1)$ is differentiable at an asymptotic direction $\alpha$ if and only if it is differentiable at the point $\mathfrak{l}^{-1}(1)\cap\overline{\alpha}$ as a curve.
\end{defn}

\begin{The}\label{non-diff sts}
Let $(\mathbb{T}^{2},g)$ be a class A Lorentzian 2-torus and $\mathfrak{l}:\mathfrak{T}\rightarrow\mathbb{R}$ be the stable time separation. Then the unit sphere of $\mathfrak{l}|_{\mathfrak{T}^{\circ}}$ is always differentiable at each irrational asymptotic direction $\alpha$ and is differentiable at a rational asymptotic direction $\alpha$ if and only if there is a foliation of $\mathbb{T}^{2}$ whose leaves all belong to $\pi\mathscr{M}_{\alpha}^{per}$.
\end{The}

\section{Geodesical completeness of some timelike rays}
In this section, we will show that every timelike ray (line) with an asymptotic direction $\alpha\in(m^{-},m^{+})$ is geodesically complete. Some priori estimates of ``uniform'' timelike vectors are needed.

\begin{defn}\label{tb&lb}
Denote by $[g]$ the conformal class of the Lorentzian metric $g$ sharing the same time orientation, we define the following subbundles of T$\mathbb{T}^{2}$:
\begin{enumerate}
  \item Time$(\mathbb{T}^{2},[g]):=\{$ future directed timelike vectors in $(\mathbb{T}^{2},[g])\}$.

  \item Light$(\mathbb{T}^{2},[g]):=\{$ future directed lightlike vectors in $(\mathbb{T}^{2},[g])\}$.
\end{enumerate}
We shall denote Time$(\mathbb{T}^{2},[g])_{p}$ and Light$(\mathbb{T}^{2},[g])_{p}$ the fibres of Time$(\mathbb{T}^{2},[g])$ and Light$(\mathbb{T}^{2},[g])$ over $p\in \mathbb{T}^{2}$ respectively.

For any small $\epsilon>0$, we define:
\begin{enumerate}\label{epsilon-tb}
  \item Time$^{\epsilon}(\mathbb{T}^{2},[g]):=\{V\in$ Time$(\mathbb{T}^{2},[g])|$ dist$(V,$ Light$(\mathbb{T}^{2},[g]))\geq\epsilon|V|\}$.

  \item Time$^{1,\epsilon}(\mathbb{T}^{2},g):=\{V\in$ Time$^{\epsilon}(\mathbb{T}^{2},[g])|g(V,V)=-1\}$.
\end{enumerate}
Be careful that the bundle Time$^{1,\epsilon}(\mathbb{T}^{2},g)$ depends on $g$, not just the conformal class $[g]$.
\end{defn}

\begin{Rem}
From the definition, Time$^{\epsilon}(\mathbb{T}^{2},[g])$ and Time$^{1,\epsilon}(\mathbb{T}^{2},g)$ are smooth subbundles of Time$(\mathbb{T}^{2},[g])$. In addition, the fibres of Time$^{\epsilon}(\mathbb{T}^{2},[g])$ are convex in each tangent space of $\mathbb{T}^{2}$.
\end{Rem}

\begin{The}\label{cpt}
There exists a constant $\mathrm{K}(g,g_{R})$ such that for any $V\in$Time$^{1,\epsilon}(\mathbb{T}^{2},g)$, we have $|V|\leq\frac{\mathrm{K}(g,g_{R})}{\epsilon}$. In other words, the smooth bundle Time$^{1,\epsilon}(\mathbb{T}^{2},g)$ is compact.
\end{The}

The proposition follows from an elementary lemma in linear algebra. To state the lemma, let $E^{2}$ be a two dimensional $\mathbb{R}$-vector space and $G(\cdot,\cdot)$ (resp. $\langle\cdot,\cdot\rangle$) a scalar product with signature $1$ (resp. $0$) on $E^{2}$. Denote by $|\cdot|$ (resp. dist$(\cdot,\cdot)$) the norm (resp. metric) induced by $\langle\cdot,\cdot\rangle$. Since $G$ is an indefinite scalar product on $E^{2}$, we could define
$$
\text{Light}(E^{2},G):=\{V\in E^{2}\setminus\{0\}|G(V,V)=0\}.
$$

\begin{Lem}\label{bound}
For every $\epsilon\in(0,1)$ and every
$$
V\in\{G(V,V)=-1|\text{ dist }(V,\text{ Light }(E^{2},G))\geq\epsilon|V|\},
$$
there is a constant $K$ depending only on $G$ and $\langle\cdot,\cdot\rangle$ such that $|V|\leq\frac{K}{\epsilon}$.
\end{Lem}
\noindent\textit{Proof of Lemma \ref{bound}}. For $V\in\{G(V,V)=-1|$ dist$(V,$ Light$(E^{2},G))\geq\epsilon|V|\}$, there exist two linearly independent $G$-lightlike vectors $X_{1}$ and $X_{2}$ such that
\begin{itemize}
  \item $|X_{i}|=1,i=1,2$;
  \item $V=\lambda X_{1}+\mu X_{2}, \lambda>0,\mu>0$.
\end{itemize}
By $G(V,V)=-1$, we have
\begin{equation}\label{a}
\lambda\mu=-\frac{1}{2G(X_{1},X_{2})}.
\end{equation}
From dist$(V,$ Light$(E^{2},G))\geq\epsilon|V|$, we get that
\begin{eqnarray*}
&&|V|^{2}-|\langle X_{i},V\rangle|^{2}\\
&=&|V-\langle X_{i},V\rangle X_{i}|^{2}\\
&\geq&\text{dist}^{2}(V,\text{Light}(E^{2},G))\\
&\geq&\epsilon^{2}|V|^{2},i=1,2.
\end{eqnarray*}
Thus hold the inequalities
\begin{equation}\label{b}
(1-\epsilon^{2})|V|^{2}\geq|\langle X_{i},V\rangle|^{2},i=1,2.
\end{equation}
Together with equalities
\begin{equation}\label{c}
\begin{split}
|\langle X_{1},V\rangle|^{2}&=|V|^{2}-\mu^{2}(1-|\langle X_{1},X_{2}\rangle|^{2}),\\
|\langle X_{2},V\rangle|^{2}&=|V|^{2}-\lambda^{2}(1-|\langle X_{1},X_{2}\rangle|^{2}),
\end{split}
\end{equation}
we obtain
\begin{align*}
\epsilon^{2}|V|^{2}&\leq\lambda^{2}(1-|\langle X_{1},X_{2}\rangle|^{2}),\\
\epsilon^{2}|V|^{2}&\leq\mu^{2}(1-|\langle X_{1},X_{2}\rangle|^{2}).
\end{align*}
Multiplying the above two inequalities and using Equality \eqref{a}, one deduces that
$$
|V|\leq\frac{1}{\epsilon}\sqrt{\frac{1-|\langle X_{1},X_{2}\rangle|^{2}}{2|G(X_{1},X_{2})|}}.
$$
By choosing $K=\sqrt{\frac{1-|\langle X_{1},X_{2}\rangle|^{2}}{2|G(X_{1},X_{2})|}}$, we complete the proof of Lemma \ref{bound}. \qed

Now we can prove Theorem \ref{cpt} as follows. For a point $p\in(\mathbb{T}^{2},g)$, take $E^{2},G(\cdot,\cdot),\langle\cdot,\cdot\rangle$ in Lemma \ref{bound} to be $T_{p}\mathbb{T}^{2},g|_{p},g_{R}|_{p}$ respectively. One could easily see that Light$(\mathbb{T}^{2},[g])_{p}$ is part of Light$(E^{2},G)$ and Time$^{1,\epsilon}(\mathbb{T}^{2},g)_{p}$ is also part of
$$
\{V|G(V,V)=-1,\text{ dist }(V,\text{ Light }(E^{2},G))\geq\epsilon|V|\}
$$
w.r.t. Light$(\mathbb{T}^{2},[g])_{p}$.

Let $\{X_{i}(p)\}_{i=1,2}$ be two future-directed smooth lightlike vector fields defined in Section 1. By Lemma \ref{bound} one concludes that for every $\epsilon\in(0,1)$ and every $V\in$Time$^{1,\epsilon}(\mathbb{T}^{2},g)_{p}$, there exists a constant $K(p)>0$ such that $|V|_{p}\leq\frac{K(p)}{\epsilon}$. From the proof of Lemma \ref{bound}, we have that
\begin{equation}
K(p)=\sqrt{\frac{1-|g_{R}(p)(X_{1}(p),X_{2}(p))|^{2}}{2|g(p)(X_{1}(p),X_{2}(p))|}}.
\end{equation}
By the smoothness of certain metrics and vector fields in the above formula, $K(p)$ is also smooth on $\mathbb{T}^{2}$. This shows that $\mathrm{K}(g,g_{R}):=\max_{p\in \mathbb{T}^{2}}K(p)$ always exists and the first assertion of Theorem \ref{cpt} follows from $|V|_{p}\leq\frac{K(p)}{\epsilon}\leq\frac{\mathrm{K}(g,g_{R})}{\epsilon}$.

The second assertion follows directly from the first one since by the above discussion, $\text{Time}^{1,\epsilon}(\mathbb{T}^{2},g)$, as a closed subset of $\{V\in T\mathbb{T}^{2}||V|\leq\frac{\mathrm{K}(g,g_{R})}{\epsilon}\}$, must be a compact subset of $T\mathbb{T}^{2}$.\qed

We shall state a direct corollary of Theorem \ref{cpt} that will be frequently used. By a pregeodesic, we mean a causal $g$-geodesic parametrized by $g_{R}$-arc length.

\begin{Cor}\label{cpl}
Let $\gamma:\mathbb{R}_{+}$ ($\mathbb{R}$) $\rightarrow(\mathbb{T}^{2},g)$ be a future inextendible (inextendible) timelike pregeodesic. If there exists an $\epsilon>0$, such that for every $t\in\mathbb{R}_{+}$ ($\mathbb{R}$), $\dot{\gamma}(t)\in$Time$^{\epsilon}(\mathbb{T}^{2},[g])_{\gamma(t)}$, then $\gamma$ can be parametrized to be a future complete (complete) $g$-geodesic.
\end{Cor}

\begin{proof}
Let $\gamma^{g}:I\rightarrow(\mathbb{T}^{2},g)$ be the reparametrization of $\gamma$ using $g$-arc length. By the condition we obtain that
\begin{equation}
\cup_{t\in I}(\gamma^{g}(t),\dot{\gamma^{g}}(t))\subseteq\text{Time}^{1,\epsilon}(\mathbb{T}^{2},g).
\end{equation}
Using Theorem \ref{cpt}, the smooth bundle $\text{Time}^{1,\epsilon}(\mathbb{T}^{2},g)$ is compact. Since geodesic flow for $(\mathbb{T}^{2},g)$ is a flow generated by a vector field $X_{g}$ defined on $T\mathbb{T}^{2}$, we know that $(\gamma^{g},\dot{\gamma^{g}})$ is an integral curve of $X_{g}$ that stays in a compact set. Thus by the fundamental knowledge of ODEs theory, if $\gamma^{g}$ is future inextendible (inextendible), its domain is $\mathbb{R}_{+}$ ($\mathbb{R}$).
\end{proof}

Next, we relate the above result to the global behavior of some pregeodesics.

\begin{defn}\label{epsilon-cone}
For $\epsilon>0$, set $\mathfrak{T}^{\epsilon}:=\{h\in\mathfrak{T}|$ dist$_{\|\cdot\|}(h,\partial\mathfrak{T})\geq\epsilon\|h\|\}$. We call $\mathfrak{T}^{\epsilon}$ a strict $\epsilon$-cone. It is obvious that $\mathfrak{T}^{\epsilon}\subseteq\mathfrak{T}^{\circ}$.
\end{defn}

\begin{Pro}\label{asym'}
Let $\gamma:\mathbb{R}_{+}\rightarrow(\mathbb{T}^{2},g)$ be a timelike pregeodesic with asymptotic direction $\alpha\in(m^{-},m^{+})$ and $\tilde{\gamma}$ be a lift of $\gamma$ to $(\mathbb{R}^{2},g)$. Then there exist two positive constants $\epsilon(\alpha,g,g_{R}),\mathrm{T}(\gamma,\alpha,g,g_{R})$ such that for any $T_{1},T_{2}\in\mathbb{R}_{+}, T_{2}-T_{1}>\mathrm{T}$, we have
$$
\tilde{\gamma}(T_{2})-\tilde{\gamma}(T_{1})\in\mathfrak{T}^{\epsilon}.
$$
\end{Pro}

\begin{proof}
By Definition \ref{asym}, there exist a half line $\overline{\alpha}$ and a number $D(\gamma)>0$ depending only on $\gamma$ such that for any $0\leq T_{1}<T_{2}$, dist$_{\|\cdot\|}(\tilde{\gamma}(T_{2})-\tilde{\gamma}(T_{1}),\overline{\alpha})\leq D(\gamma)$.

Since $\alpha\in(m^{-},m^{+})\subseteq\mathfrak{T}^{\circ}$ and $\partial\mathfrak{T}=\overline{m}^{-}\cup\overline{m}^{+}$,
$$
\theta_{\alpha}:=\{\text{dist}(\alpha,\overline{m}^{-}\cup\overline{m}^{+})\}>0.
$$

Because $\gamma$ is a pregeodesic, $\tilde{\gamma}$ is also parametrized by $g_{R}$-arc length. By Corollary \ref{causal} and Theorem \ref{std}, we could get that
\begin{eqnarray}\label{ieq:1}
&&\|\tilde{\gamma}(T_{2})-\tilde{\gamma}(T_{1})\|\\
&\geq&\text{d}_{R}(\tilde{\gamma}(T_{1}),\tilde{\gamma}(T_{2}))-|\text{std}(g_{R})|\\
&\geq&\frac{T_{2}-T_{1}}{C_{g,g_{R}}}-|\text{std}(g_{R})|.
\end{eqnarray}

Since there is an element $h\in\overline{\alpha}$ such that dist$_{\|\cdot\|}(\tilde{\gamma}(T_{2})-\tilde{\gamma}(T_{1}),h)\leq D(\gamma)$, we have
$$
\|h\|\geq\|\tilde{\gamma}(T_{2})-\tilde{\gamma}(T_{1})\|-D(\gamma).
$$
Thus,
\begin{eqnarray*}
&&\text{dist}_{\|\cdot\|}(\tilde{\gamma}(T_{2})-\tilde{\gamma}(T_{1}),\overline{m}^{\pm})\\
&\geq&\text{dist}_{\|\cdot\|}(h,\overline{m}^{\pm})-D(\gamma)\\
&=&\|h\|\text{dist}_{\|\cdot\|}(\alpha,\overline{m}^{\pm})-D(\gamma)\\
&\geq&\theta_{\alpha}\|\tilde{\gamma}(T_{2})-\tilde{\gamma}(T_{1})\|-(\theta_{\alpha}+1)D(\gamma)\\
&\geq&\frac{1}{2}\theta_{\alpha}\|\tilde{\gamma}(T_{2})-\tilde{\gamma}(T_{1})\|
\end{eqnarray*}
if $T_{2}-T_{1}\geq T:= C_{g,g_{R}}(2\frac{\theta_{\alpha}+1}{\theta_{\alpha}}D(\gamma)+|\text{std}(g_{R})|)$ by Inequality \ref{ieq:1}. We choose $\epsilon=\theta_{\alpha}$ to complete the proof.
\end{proof}

\begin{Rem}
Due to Theorem \ref{structure for lines}, one deduces that, for all timelike lines with asymptotic directions in $(m^{-},m^{+})$, the positive constant $\mathrm{T}$ given by Proposition \ref{asym'} depends only on $\alpha,g$ and $g_{R}$.
\end{Rem}

Since S. Suhr proved in \cite{Su3} that any class A Lorentzian 2-torus is of class $A_{1}$, we could use Proposition \ref{A1-1}, together with Corollary \ref{asym} and Proposition \ref{asym'},  to obtain the following corollary.

\begin{Cor}\label{rl cpl}
If a timelike pregeodesic $\gamma:I\rightarrow(\mathbb{T}^{2},g)$ is a ray (line) with asymptotic direction $\alpha\in(m^{-},m^{+})$, then it could be parametrized by affine parameter as a future complete (complete) $g$-geodesic.
\end{Cor}

\section{Rays with a fixed timelike asymptotic direction}
This section concerns intersection properties of timelike rays on $(\mathbb{R}^{2},g)$, which have an asymptotic direction in $(m^{-},m^{+})$. All curves we consider in this section lie on $(\mathbb{R}^{2},g)$ and are parametrized by the arc-length w.r.t. $g_{R}$.

We call a timelike maximizer $\gamma:I\rightarrow(\mathbb{R}^{2},g)$ a timelike maximal segment if $I=[a,b]$ is a closed interval. Three obvious facts we shall use (or implicitly use) in this section are:
\begin{itemize}
  \item Two inextendible causal curves with different asymptotic directions must have a intersection.
  \item Two timelike maximal segments cannot intersect twice, except that the intersections occured  at their initial points and end points respectively. Since the Riemannian version of this fact was first used by M. Morse in his celebrated paper \cite{M}, we call this fact Morse's lemma.
  \item Every maximizer can be parametrized as a smooth geodesic. Moreover, if a timelike curve $\gamma:[a,b]\rightarrow(\mathbb{R}^{2},g)$ has a corner in the interior, then there exists a constant $\varepsilon_{0}>0$ depending only on $\gamma$ itself such that any maximal segment $\zeta$ connecting $\gamma(a)$ with $\gamma(b)$ satisfies $L^{g}(\zeta)-L^{g}(\gamma)\geq\varepsilon_{0}>0$. This is called curve lengthening lemma.
\end{itemize}

Now we shall define several types of the intersections of timelike lines and maximal timelike segments.

First, let $\gamma$ be a timelike line that belongs to $\mathscr{M}:=\cup_{\alpha\in(m^{-},m^{+})}\mathscr{M}_{\alpha}$ and $\zeta:[a,b]\rightarrow\mathbb{R}^{2}$ be a maximal timelike pregeodesic segment. If the timelike maximizers $\zeta$ and $\gamma$ intersect at some interior point of $\zeta$, then either $Im(\zeta)\subseteq Im(\gamma)$ or the intersection must be transversal since they could be reparametrized as geodesics w.r.t. $g$. We shall call the intersection in the first (second) case a trivial (non-trivial) intersection. We shall focus on non-trivial intersections. Notice that the above intersections do not include the case that $\zeta$ and $\gamma$ intersect at some endpoint of $\zeta$.

By Morse's lemma, $\zeta$ and $\gamma$ intersect non-trivially at most once. Recall that since $\gamma$ is a timelike line, it divides $\mathbb{R}^{2}$ into two connected components $U^{+}$ and $U^{-}$. If $\zeta$ and $\gamma$ have a non-trivial intersection at an interior point of $\zeta$, then $\zeta(a)$ and $\zeta(b)$ must lie in different components of $U^{+}\cup U^{-}=\mathbb{R}^{2}\setminus Im(\gamma)$. If they do not intersect each other, we have $Im(\zeta)\subseteq U^{+}$ (or $U^{-}$).

The above discussions lead to the following definition (in this definition, the relations $<,>,\leq,\geq$ are defined in Definition \ref{<,>}).

\begin{defn}\label{cross types}
If a timelike maximal segment $\zeta$ and a line $\gamma\in\mathscr{M}$ intersect non-trivially at some $\zeta(t),t\in(a,b)$, then:
\begin{itemize}
  \item Either $\zeta(a)<\gamma,\zeta(b)>\gamma$;
  \item Or $\zeta(a)>\gamma,\zeta(b)<\gamma$.
\end{itemize}
In the first case, we say the intersection of $\zeta$ and $\gamma$ is of type 1; in the second case, we say the intersection of $\zeta$ and $\gamma$ is of type 2.

If $\zeta(t)\leq(\geq)\gamma$ for all $t\in[a,b]$, we say $\zeta\leq(\geq)\gamma$. If $\zeta(t)<(>)\gamma$ for all $t\in[a,b]$, we say $\zeta<(>)\gamma$. These two relations can be easily extended to the cases that $\zeta$ is a ray or a line by defining that $\zeta\leq(<)\gamma$ if $\zeta|_{[a,b]}\leq(<)\gamma$ holds for any closed interval $[a,b]$ in the domain of $\zeta$.
\end{defn}

\begin{Rem}\label{transversal}
If a timelike maximal segment $\zeta:[a,b]\rightarrow(\mathbb{R}^{2},g)$ (here $a<0<b$) and a timelike line $\gamma$ have a non-trivial intersection at $\zeta(0)=\gamma(0)$, then the intersection is of type 1 (type 2) if and only if $\{\dot{\gamma}(0),\dot{\zeta}(0)\}$ is positively (negatively) oriented.
\end{Rem}

\begin{Lem}\label{stable types}
Let $(\mathbb{R}^{2},g)$ be the Abelian cover of a class A Lorentzian 2-torus and $\gamma_{k}$ be a sequence of timelike lines converges to a causal line $\gamma$ in the $C^0$-topology. If a timelike maximal segment $\zeta:[a,b]\rightarrow(\mathbb{R}^{2},g)$ intersects all $\gamma_{k}$ with intersections of type 1 (type 2), and $\zeta$ also intersect $\gamma$ nontrivially, then $\gamma_{k}$ converges to $\gamma$ in the $C^1$-topology and the intersection of $\zeta$ and $\gamma$ is also of type 1(type 2).
\end{Lem}

\begin{proof}
For the timelike segment  $\zeta:[a,b]\rightarrow(\mathbb{R}^{2},g)$, we assume that $a<0<b$, $\zeta$ and $\gamma$ intersect at $\zeta(0)=\gamma(0)$ without loss of generality. By Morse's lemma, the intersection point of $\zeta$ and $\gamma_{k}$ ($\gamma$) is unique. So if we denote their intersection point by $\zeta(a_{k})=\gamma_{k}(b_{k})$, then by the uniqueness of intersection, $a_{k},b_{k}\rightarrow0$. As S. Suhr proved in \cite[Proposition 4.12]{Su2}, all pregeodesics satisfy a smooth second order ODE system that defines a flow on the tangent bundle. Thus the $C^{0}$-convergence of $\gamma_{k}$ to $\gamma$ and the fact that $a_{k}\rightarrow0$ imply $\dot{\gamma_{k}}(a_{k})\rightarrow\dot{\gamma}(0)$ w.r.t. the usual topology on T$\mathbb{R}^{2}$. Since $\gamma_{k}$ and $\gamma$ are all parametrized by $g_{R}$-arc length, $\gamma_{k}$ converges to $\gamma$ in the $C^{1}$-topology. Obviously we also have $\dot{\zeta}(b_{k})\rightarrow\dot{\zeta}(0)$.

By the transversality of the intersection of $\zeta$ and $\gamma$, $\dot{\zeta}(0)$ and $\dot{\gamma}(0)$ are linearly independent. Thus we obtain
$$
\lim_{k\rightarrow\infty}\det({\dot{\gamma_{k}}(a_{k}),\dot{\zeta}(b_{k})})=\det({\dot{\gamma}(0),\dot{\zeta}(0)})\neq0.
$$
Then the lemma follows from Remark \ref{transversal}.
\end{proof}

We are now ready to formulate the main lemma of this section.
\begin{Lem}\label{main lemma}
Let $\zeta:[a,b]\rightarrow(\mathbb{R}^{2},g)$ be a subsegment of a timelike ray with asymptotic direction in $(m^{-},m^{+})$. Then there exist an asymptotic direction $\alpha\in(m^{-},m^{+})$ and two neighboring elements $\underline{\zeta},\overline{\zeta}$ in $\mathscr{M}_{\alpha}^{rec}$ (if $\alpha$ is irrational) resp. in $\mathscr{M}_{\alpha}^{per}$ (if $\alpha$ is rational) such that $\underline{\zeta}\leq\zeta\leq\overline{\zeta}$.
\end{Lem}

This is just an adaption of the proof of \cite[Theorem 3.2]{Ba 2}. For the completeness, we still represent it here.

\begin{proof}
If there is a $\beta\in(m^{-},m^{+})$ such that $\zeta$ does not intersect any timelike line in $\mathscr{M}_{\beta}^{rec}$ (if $\beta$ is irrational) or in $\mathscr{M}_{\beta}^{per}$ (if $\beta$ is rational) at an interior point of $\zeta$, then the theorem is true for $\zeta$. Otherwise, for every $\beta\in(m^{-},m^{+})$, there exists $\gamma\in\mathscr{M}_{\beta}$ such that $\zeta$ intersects $\gamma$ at some interior point. Assume this intersection is of type 1, then the following subset of asymptotic directions is non-empty:

$B=\{\beta\in(m^{-},m^{+})$  is irrational $|$ There exists $ \gamma\in\mathscr{M}_{\beta}^{rec}$ such that the intersection of $\zeta$ and $\gamma$ is of type 1 $\}$.

Since $\zeta$ is a subsegment of some timelike ray with an asymptotic direction in $(m^{-},m^{+})$, then $\alpha:=\sup B$ clearly exists and belongs to $(m^{-},m^{+})$. We argue by a contradiction that $\alpha$ satisfies the lemma's requirement.

If $\alpha$ is irrational, assume $\zeta$ intersects transversally some $\gamma\in\mathscr{M}_{\alpha}^{rec}$. Since $\gamma$ is the $C^{0}$-limit of a sequence $\gamma_{k}\in\mathscr{M}_{\alpha_{k}}^{rec}$ with $\alpha_{k}>\alpha$, then by Remark \ref{transversal}, $\zeta$ intersects almost every  $\gamma_{k}$ at some interior point and all intersections are nontrivial. By the definition of $\alpha$, the intersection of $\zeta$ and $\gamma_{k}$ is of type 2 for every sufficiently large $k$. By Lemma \ref{stable types}, we obtain that the intersection of $\zeta$ and $\gamma$ is also of type 2. By Theorem \ref{structure for lines}, we deduce that
\begin{equation}\label{star}
\text{If }\gamma_{1}\in\mathscr{M}_{\alpha},\zeta(a)\leq\gamma_{1}, \text{then }\zeta(b)<\gamma_{1}.
\end{equation}
Then $\alpha$ does not belong to $B$. On the other hand, there is a sequence $\beta_{k}\in B$ with $\lim\beta_{k}=\alpha$ and by Theorem \ref{structure for lines}, a sequence $\eta_{k}\in\mathscr{M}_{\beta_{k}}^{rec}$ such that the intersections of $\zeta$ and $\eta_{k}$ are of type 1. We translate the parameter of  $\eta_{k}$ such that $\eta_{k}(0)$ is the intersection point of $\zeta$ and $\eta_{k}$. Since $Im(\zeta)$ is compact and $\eta_{k}$ is parametrized by $g_{R}$-length, Ascoli-Arzela Theorem shows that $\{\eta_{k}\}$ is compact w.r.t. the  $C^{0}$ topology. Thus there exists a subsequence of $\eta_{k}$ converges to some $\eta\in\mathscr{M}_{\alpha}$. The $C^{0}$-convergence of $\eta_{k}$ implies that such $\eta$ satisfies $\zeta(a)\leq\eta$ and $\zeta(b)\geq\eta$ (if it is not so, one could deduce that the intersection of $\zeta$ and $\eta_{k}$ is of type 2, which is absurd by our assumption), which contradicts \eqref{star}. So we have proved the lemma when $\alpha$ is irrational.

If $\alpha$ is rational, again we assume that $\zeta$ intersects some $\gamma\in\mathscr{M}_{\alpha}^{per}$ nontrivially. As above, we get that the intersection of $\zeta$ and $\gamma$ is of type 2, so
\begin{equation}\label{ieq:0}
\zeta(a)>\gamma,\zeta(b)<\gamma.
\end{equation}
Since $\alpha$ is rational, it does not belong to $B$. Also by the same argument as above, we could find an element $\eta\in\mathscr{M}_{\alpha}$ such that
\begin{equation}\label{ieq:10}
\zeta(a)\leq\eta,\zeta(b)\geq\eta.
\end{equation}
If $\eta$ and $\gamma$ do not intersect each other, then either $\eta>\gamma$ or $\eta<\gamma$. By Inequality \ref{ieq:10}, both cases contradict Inequality \ref{ieq:0}. So we get that $\gamma\in\mathscr{M}_{\alpha}^{per}$ and $\eta\in\mathscr{M}_{\alpha}$ intersect, this contradicts Theorem \ref{structure for lines}. This completes the proof of this lemma.
\end{proof}

\begin{The}\label{nb lines for a ray}
For every $\zeta\in\mathscr{R}_{\alpha} \text{ and }\alpha\in(m^{-},m^{+})$, there exist two neighboring lines $\underline{\zeta},\overline{\zeta}$ in $\mathscr{M}_{\alpha}^{per}\text{ if }\alpha\in\mathbb{Q}$ (in $\mathscr{M}_{\alpha}^{rec}\text{ if }\alpha\in\mathbb{R}\setminus\mathbb{Q}$) such that $\underline{\zeta}\leq\zeta\leq\overline{\zeta}$.
\end{The}

\begin{proof}
Applying Lemma \ref{main lemma} to $\zeta_{j}:=\zeta|_{[0,j]}$, we know that there exists a sequence $\{\alpha_{j}\}$ of asymptotic directions such that there are two neighboring elements, $\overline{\zeta}_{j}$ and $\underline{\zeta}_{j}$, in $\mathscr{M}_{\alpha_{j}}^{rec}$ (or $\mathscr{M}_{\alpha_{j}}^{per}$) satisfying $\underline{\zeta}_{j}\leq\zeta_{j}\leq\overline{\zeta}_{j}$.

\textbf{Claim}: $\alpha_{j}\rightarrow\alpha$ in $SH_{1}(\mathbb{T}^{2},\mathbb{R})$ w.r.t. the topology defined in Definition \ref{asym} as $j\rightarrow\infty$.

If the claim has been proved, we can easily complete the proof by the following argument. Suppose the theorem is not true, there exist an element $\gamma\in\mathscr{M}_{\alpha}^{rec}$ and a large $N$ such that $\zeta_{N}$ intersects $\gamma$ nontrivially. Since $\alpha_{j}\rightarrow\alpha$, there exists a sequence of lines $\gamma_{j}\in\mathscr{M}_{\alpha_{j}}^{rec}$ (or $\mathscr{M}_{\alpha_{j}}^{per}$) converges to $\gamma$ uniformly on compact intervals. Then we have $\zeta_{N}$ also intersects $\gamma_{j}$ nontrivially for sufficiently large $j$, this contradicts Lemma \ref{main lemma}.\qed

\noindent\textit{Proof of the Claim}: Since $\overline{\zeta}_{j},\underline{\zeta}_{j}$ are neighboring elements in $\mathscr{M}_{\alpha_{j}}$ and $\zeta_{j}$ is contained in the strip bounded by $\overline{\zeta}_{j}$ and $\underline{\zeta}_{j}$, there are $p_{j},q_{j}\in\overline{\zeta}_{j}$ such that
\begin{equation}
\begin{split}
d_{R}(\zeta(0),p_{j})\leq\text{diam}(\mathbb{T}^{2},g_{R}),\\
d_{R}(\zeta(j),q_{j})\leq\text{diam}(\mathbb{T}^{2},g_{R}).
\end{split}
\end{equation}
By Theorem \ref{structure for lines}, there is a constant $D(g,g_{R})$ such that dist$_{\|\cdot\|}(q_{j}-p_{j},\overline{\alpha_{j}})\leq D$ for any positive integer $j$. Together with Theorem \ref{std}, there is a constant $D^{\prime}(g,g_{R})$ and for each $j$, a vector $h_{j}\in\overline{\alpha_{j}}$ such that
\begin{equation}\label{ieq:2}
\text{dist}_{\|\cdot\|}(\zeta(j)-\zeta(0),h_{j})\leq D^{\prime}.
\end{equation}

Since $\zeta\in\mathscr{R}_{\alpha}$, there is a constant $D(\gamma)$ and for each $j$, a vector $\bar{h}_{j}\in\overline{\alpha}$ such that
\begin{equation}\label{ieq:4}
\text{dist}_{\|\cdot\|}(\zeta(j)-\zeta(0),\bar{h}_{j})\leq D(\gamma).
\end{equation}
As $\zeta$ is a timelike ray, $\|\zeta(j)-\zeta(0)\|\rightarrow\infty$ as $j\rightarrow\infty$. Dividing two sides of \eqref{ieq:2} and \eqref{ieq:4} by $\|\zeta(j)-\zeta(0)\|$, we got $\alpha_{j}\rightarrow\alpha$ as $j\rightarrow\infty$.
\end{proof}

We state a useful corollary of Theorem \ref{nb lines for a ray}.
\begin{Cor}\label{ray osc}
There exists a constant $B(g,g_{R})$ such that for any $\alpha\in(m^{-},m^{+})$ and any timelike ray $\zeta\in\mathscr{R}_{\alpha}$,
\begin{equation}
\text{dist}_{\|\cdot\|}(\zeta(t)-\zeta(s),\overline{\alpha})\leq B(g,g_{R})
\end{equation}
for all $s\leq t$ contained in the domain of $\zeta$.
\end{Cor}

\begin{proof}
One could find two points $p,q$ on $\overline{\zeta}$ (or $\underline{\zeta}$) such that
\begin{equation}
\begin{split}
d_{R}(\zeta(s),p)\leq\text{diam}(\mathbb{T}^{2},g_{R}),\\
d_{R}(\zeta(t),q)\leq\text{diam}(\mathbb{T}^{2},g_{R}).
\end{split}
\end{equation}
By Theorem 8.1, together with the fact
\begin{equation}
\text{dist}_{\|\cdot\|}(q-p,\overline{\alpha})\leq D(g,g_{R}),
\end{equation}
one easily deduces the conclusion.
\end{proof}

\begin{Lem}\label{dom}
For $\alpha\in(m^{-},m^{+})$, let $\gamma:\mathbb{R}\rightarrow(\mathbb{R}^{2},g)$ be a timelike line in $\mathscr{M}_{\alpha}$. Then
\begin{equation}
I^{+}(\gamma)=\cup_{k\in\mathbb{N}}I^{+}(\gamma(-k))=\mathbb{R}^{2}.
\end{equation}
\end{Lem}

\begin{Rem}\label{dom'}
By the same method, we could prove that under the same conditions as above, $I^{-}(\gamma)=\cup_{k\in\mathbb{N}}I^{-}(\gamma(k))=\mathbb{R}^{2}$.
\end{Rem}

\begin{proof}
The first equality is obvious since $\gamma$ is future directed and timelike. More precisely, we have $I^{+}(\gamma(s))\subseteq I^{+}(\gamma(t))$ if $s\geq t$. Thus $I^{+}(\gamma(s))\subseteq I^{+}(\gamma([s]))$ holds, where $[s]$ denotes the integer part of $s$.

Fix arbitrarily an point $p\in\mathbb{R}^{2}$. To prove the second equality, we shall use the smooth lightlike vector fields $X^{\pm}$ defined in Section 1. Choose one of their integral curves $\eta_{p}$ through $p$ such that $\eta_{p}$ and $\gamma$ cross at $\gamma(s)$ and $p\in J^{+}(\gamma(s))$. By the definition of asymptotic directions, such an $\eta_{p}$ does exist. Since $\gamma(s)\in I^{+}(\gamma([s]-1))$, we obtain that $p\in I^{+}(\gamma([s]-1))$. This completes the proof.
\end{proof}

\begin{Lem}\label{no intersection}
For $\alpha\in(m^{-},m^{+})$, let $\zeta:\mathbb{R}_{+}\rightarrow(\mathbb{R}^{2},g)$ be a timelike ray in $\mathscr{R}_{\alpha}$ that is not a subray of some timelike line in $\mathscr{M}_{\alpha}$. If $\zeta$ is asymptotic to $\gamma\in\mathscr{M}_{\alpha}$ in the future direction (defined in Definition \ref{asy}), then $\zeta$ and $\gamma$ does not intersect.
\end{Lem}

\begin{proof}
Without loss of generality, suppose $\zeta(0)=\gamma(0)$. Since $\zeta$ is not a subray of any timelike line in $\mathscr{M}_{\alpha}$, then for any $T>0$, the conjunction curve $\gamma|_{[-1,0]}\ast\zeta|_{[0,T]}$ has a corner at $\zeta(0)$. By the curve lengthening lemma, for any fixed $T_{0}>0$, there is an $\varepsilon_{0}>0$ such that for any $T\geq T_{0}>0$ and any timelike maximal segment $\eta$ connecting $\gamma(-1)$ with $\zeta(T)$,
\begin{equation}\label{ieq:3}
L^{g}(\eta)\geq L^{g}(\gamma|_{[-1,0]}\ast\zeta|_{[0,T]})+\varepsilon_{0}.
\end{equation}

By the maximality of $\zeta$ and Proposition \ref{A1-2} (since a class A Lorentzian 2-torus is of class $A_{1}$, see \cite{Su3}), for any $\epsilon>0$, we could find $p_{\epsilon}$ on $\gamma$ and $T_{\epsilon}>T_{0}$ such that
\begin{equation}\label{ieq:5}
d(\zeta(0),\zeta(T_{\epsilon}))=L^{g}(\zeta|_{[0,T_{\epsilon}]})>d(\zeta(0),p_{\epsilon})-\epsilon,
\end{equation}
and
\begin{equation}\label{ieq:7}
d_{R}(\zeta(T_{\epsilon}),p_{\epsilon})<\epsilon.
\end{equation}
Thus we obtain by Inequality \ref{ieq:5} that
\begin{eqnarray*}
&&L^{g}(\gamma|_{[-1,0]}\ast\zeta|_{[0,T_{\epsilon}]})\\
&>&d(\gamma(-1),\gamma(0))+d(\zeta(0),p_{\epsilon})-\epsilon\\
&=&d(\gamma(-1),p_{\epsilon})-\epsilon.
\end{eqnarray*}
By Proposition \ref{asym'}, Proposition \ref{A1-2} and Inequality \ref{ieq:3}, there is a constant $C>0$ such that for the timelike maximal segment $\eta_{\epsilon}$ connecting $\gamma(-1)$ with $\zeta(T_{\epsilon})$,
\begin{eqnarray*}
&&d(\gamma(-1),p_{\epsilon})\\
&>&d(\gamma(-1),\zeta(T_{\epsilon}))-C\epsilon\\
&=&L^{g}(\eta_{\epsilon})-C\epsilon\\
&\geq&L^{g}(\gamma|_{[-1,0]}\ast\zeta|_{[0,T_{\epsilon}]})+\varepsilon_{0}-C\epsilon\\
&>&d(\gamma(-1),p_{\epsilon})+\varepsilon_{0}-(C+1)\epsilon,
\end{eqnarray*}
where $C$ (depending only on $g$ and $g_{R}$) is the Lipschitz constant of $d(\cdot,\cdot)$ w.r.t. the metric $d_{R}$. By choosing $\epsilon$ such that $(C+1)\epsilon<\varepsilon_{0}$, we obtain a contradiction.
\end{proof}

\begin{Rem}\label{non-crossing}
One can see from the proof that if two rays with the  same asymptotic direction are asymptotic in future direction, then they either has a common initial point or does not intersect.
\end{Rem}

We could state the following two theorems describing the structure of lifted rays in $\mathscr{R}_{\alpha}$ for $\alpha\in(m^{-},m^{+})$.

\begin{The}\label{irrational}
Let $\alpha\in(m^{-},m^{+})$ be an irrational timelike asymptotic direction, $\zeta\in\mathscr{R}_{\alpha}$  be a timelike ray which is not a subray of some timelike line in $\mathscr{M}_{\alpha}$, then there exist two neighboring timelike lines $\underline{\zeta},\overline{\zeta}$ in $\mathscr{M}_{\alpha}^{rec}$ such that $\underline{\zeta}<\zeta<\overline{\zeta}$. On the other hand, for any irrational $\alpha\in(m_{-},m_{+})$ and any $q\in\mathbb{R}^{2}$, there exists at least one timelike ray $\zeta\in\mathscr{R}_{\alpha}$ emanating from $q$.
\end{The}

\begin{proof}
The first assertion is easily proved as following. Since $\zeta:\mathbb{R}_{+}\rightarrow(\mathbb{R}^{2},g)$ is a timelike ray in $\mathscr{R}_{\alpha}$, by Theorem \ref{nb lines for a ray}, we know that there exist two neighboring timelike lines $\underline{\zeta},\overline{\zeta}$ in $\mathscr{M}_{\alpha}^{rec}$ such that $\underline{\zeta}\leq\zeta\leq\overline{\zeta}$. By Theorem \ref{structure for lines}, $\overline{\zeta}$ is asymptotic to $\underline{\zeta}$, then we deduce that $\zeta$ is also asymptotic to $\underline{\zeta}$ by Corollary \ref{causal} and the fact that $\underline{\zeta}\leq\zeta\leq\overline{\zeta}$. Thus the remaining follows directly from Remark \ref{non-crossing}.

To show the second assertion, let $\underline{\zeta},\overline{\zeta}$ in $\mathscr{M}_{\alpha}^{rec}$ be two neighboring timelike lines and $p$ contained in the interior of the strip bounded by $\underline{\zeta}$ and $\overline{\zeta}$, we only need to show that there is a ray $\zeta:\mathbb{R}_{+}\rightarrow(\mathbb{R}^{2},g)$ in $\mathscr{R}_{\alpha}$ such that $\zeta(0)=p$. By Lemma \ref{dom} and Remark \ref{dom'}, there exists a positive integer $k_{0}$ such that $\underline{\zeta}(k)\subseteq I^{+}(p)$ for all $k\geq k_{0}$. Let $x_{k}=\underline{\zeta}(k)$, we have $d(p,x_{k})<\infty$ and $\lim_{k\rightarrow\infty}d(p,x_{k})=\infty$ by the global hyperbolicity of $(\mathbb{R}^{2},g)$ and the future ($g$-geodesically) completeness of $\underline{\zeta}$. Since $(\mathbb{R}^{2},g)$ is globally hyperbolic, we parametrize the timelike maximal segment connecting $p$ with $x_{k}$ by $g_{R}$ arc-length and denote it by $\zeta_{k}$. Using \cite[Lemma 2.4]{G-H}, we know that the sequence $\{\zeta_{k}\}$ has accumulated  points w.r.t. the  $C^{0}$-topology and every limit curve $\zeta$ is a timelike ray starting at $p$. By Morse's lemma, $\underline{\zeta}\leq\zeta\leq\overline{\zeta}$, thus $\zeta\in\mathscr{R}_{\alpha}$.
\end{proof}

\begin{defn}
For a rational asymptotic direction $\alpha\in(m^{-},m^{+})$, define
$$
\mathscr{R}_{\alpha}^{per}:=\{\gamma\in\mathscr{R}_{\alpha}| \gamma\text{ is a subray of some } x\in\mathscr{M}_{\alpha}^{per}\},
$$
$$
\mathscr{R}_{\alpha}^{+}:=\{\gamma\in\mathscr{R}_{\alpha}| \text{ There exists } x\in\mathscr{M}_{\alpha}^{per} \text{ such that } \gamma<x \text{ and } \gamma \text{ is asymptotic to } x\},
$$
$$
\mathscr{R}_{\alpha}^{-}:=\{\gamma\in\mathscr{R}_{\alpha}| \text{ There exists } x\in\mathscr{M}_{\alpha}^{per} \text{ such that } \gamma>x \text{ and } \gamma \text{ is asymptotic to } x\}.
$$
\end{defn}

\begin{The}\label{rational}
For any rational asymptotic direction $\alpha\in(m^{-},m^{+})$, we have  $\mathscr{R}_{\alpha}=\mathscr{R}_{\alpha}^{per}\cup\mathscr{R}_{\alpha}^{+}\cup\mathscr{R}_{\alpha}^{-}$, i.e. if $\zeta\in\mathscr{R}_{\alpha}\setminus\mathscr{R}_{\alpha}^{per}$, then there exist two neighboring timelike lines $\underline{\zeta},\overline{\zeta}$ in $\mathscr{M}_{\alpha}^{per}$ such that $\underline{\zeta}<\zeta<\overline{\zeta}$ and $\zeta$ is either asymptotic to $\underline{\zeta}$ in the future direction or asymptotic to $\overline{\zeta}$ in the future direction. On the other hand, if $\underline{\zeta}<\overline{\zeta}$ are neighboring timelike lines in $\mathscr{M}_{\alpha}^{per}$ and $p$ is contained in the strip bounded by $\underline{\zeta}$ and $\overline{\zeta}$, then there exist at least two timelike rays in $\mathscr{R}_{\alpha}$, say $\zeta^{\pm}:\mathbb{R}_{+}\rightarrow(\mathbb{R}^{2},g)$, such that $\zeta^{+}$ ($\zeta^{-}$) is asymptotic to $\overline{\zeta}$ ($\underline{\zeta}$) in the future direction and $\zeta^{\pm}(0)=p$.
\end{The}

\begin{proof}
Let $(q,p)\in\overline{\alpha}$ be an irreducible integral homology class and $\zeta:\mathbb{R}_{+}\rightarrow(\mathbb{R}^{2},g)$ be a timelike ray in $\mathscr{R}_{\alpha}\setminus\mathscr{R}_{\alpha}^{per}$. By the class A condition on  $(\mathbb{T}^{2},g)$,  there is a  closed transversal 1-form. So we can choose an integral curve of the kernel distribution  of the closed transversal 1-form such that its lift to $\mathbb{R}^2$ passes through $\zeta(0)$. We denote the lift curve by $\eta$. Clearly $\eta$ is a smooth spacelike curve. More precisely it is  a Cauchy hypersurface for $(\mathbb{R}^{2},g)$. We denote $\eta_{k}=T_{k(q,p)}\eta$ ($k\in\mathbb{N}$), then $\zeta$ intersects every $\eta_{k}$ at a unique point $P_{k}=\zeta(a_{k})=\eta_{k}(b_{k})$ and $a_{k}\rightarrow\infty$ as $k\rightarrow\infty$.

There exists a sequence of integers $k_{n}\rightarrow\infty$ such that $T_{-k_{n}(q,p)}\zeta$ converges to a timelike line $\overline{\zeta}\in\mathscr{M}_{\alpha}$ in the $C^0$-topology. By Theorem \ref{structure for lines}, $\overline{\zeta}$ is asymptotic to some periodic line in the future direction, so we could assume $\overline{\zeta}\in\mathscr{M}_{\alpha}^{per}$. By Lemma \ref{no intersection}, $\zeta$ and $\overline{\zeta}$ do not intersect and $\zeta(0)\notin Im(\overline{\zeta})$. For every $k$, $\overline{\zeta}$ intersects $\eta_{k}$ at $Q_{k}=\zeta(kT)=\eta_{k}(b)$, here $T$ and $b$ are constants since $\overline{\zeta}$ is periodic. Without loss of generality, we  assume that $\overline{\zeta}>\zeta$ and $b_{k}<b$.

By the facts that $T_{-k_{n}(q,p)}\zeta$ converges to $\overline{\zeta}\in\mathscr{M}_{\alpha}$  in the $C^0$-topology and that $a_{k}\rightarrow\infty$, we obtain that
\begin{equation}\label{omega asymptotic}
\lim_{n\rightarrow\infty}|b_{k_{n}}-b|=0.
\end{equation}
Since $\zeta$ and $T_{(-q,-p)}\zeta$ are both timelike rays with the same asymptotic direction $\alpha$, they can intersect at most once. Thus by Equality \ref{omega asymptotic}, there exists $N$ such that $b_{k+1}>b_{k}$ for all $k\geq N$. This fact and  Equality \ref{omega asymptotic} lead to
\begin{equation}
\lim_{k\rightarrow\infty}b_{k}=b.
\end{equation}
So for every $j\geq0$, $T_{-k(q,p)}P_{k+j}\rightarrow Q_{j}$ as $k\rightarrow\infty$. By the maximality of $\zeta$ and $\overline{\zeta}$, $T_{-k(q,p)}\zeta|_{[a_{k},a_{k+1}]}$ converges to $\overline{\zeta}|_{[0,T]}$, this proves that $\zeta$ is asymptotic to $\overline{\zeta}$ in the future direction.

Now let $\underline{\zeta}:=\max\{\gamma\in\mathscr{M}^{per}_{\alpha}|\zeta>\gamma\}$, this is well defined since $\mathscr{M}^{per}_{\alpha}$ is a closed subset of $\mathbb{R}^{2}$. Then we obtain that $\zeta>\underline{\zeta}$ and $\underline{\zeta},\overline{\zeta}$ are neighbors in $\mathscr{M}_{\alpha}^{per}$.

Finally, we show the existence of $\zeta\in\mathscr{R}_{\alpha}^{+}$ with $\zeta(0)=p$ when $p$ does not belong to any periodic line in $\mathscr{M}^{per}_{\alpha}$ (the other case is completely similar). Since now there is a neighboring pair of periodic lines $\underline{\zeta},\overline{\zeta}$ such that $p$ is contained in the interior of the strip bounded by $\underline{\zeta}$ and $\overline{\zeta}$. We choose a point $\overline{\zeta}(i)\in I^{+}(p)$ (the existence is guaranteed by Lemma \ref{dom}) and define $\zeta_{k}$ ($k\in\mathbb{N}$) to be the maximal segment connecting $\zeta(0)$ with $\overline{\zeta}(i+k)$. Let $\zeta$ be a limit curve of $\zeta_{k}$ w.r.t. the $C^{0}$-topology, then $\zeta$ is a timelike ray with asymptotic direction $\alpha$ by the fact $\underline{\zeta}\leq\zeta\leq\overline{\zeta}$. To show that $\zeta$ is asymptotic to $\overline{\zeta}$ we note that there is a timelike line $\eta\in\mathscr{M}_{\alpha}^{+}$ such that $\underline{\zeta}<\eta<\overline{\zeta}$ and $\eta<\zeta(0)$. If $\zeta$ is not asymptotic to $\overline{\zeta}$ in the future direction, then it will be asymptotic to $\underline{\zeta}$. So $\zeta(T)<\eta$ for some $T>0$. Since $\zeta_{k}$ converges to $\zeta$ in the $C^{0}$-topology, we find that $\zeta_{k}(T)<\eta$ for all sufficiently large $k$. Since $\zeta_{k}(0)=\zeta(0)>\eta$ and $\overline{\zeta}(i+k)>\eta$, $\zeta_{k}$ and $\eta$ must intersect each other at least twice at their interior points. This contradicts  Morse's lemma since both $\eta$ and $\zeta_{k}$ are timelike maximizers.
\end{proof}

\section{Lorentzian Busemann function on $(\mathbb{R}^{2},g)$}
In this section, we would like to explore some concepts and properties relating to the  Lorentzian Busemann functions for rays. We also introduce the definition of Lorentzian Busemann function associated to lines. It turns out that the study of the latter objects completely cover the study of the former ones and is obviously more convenient. We are glad to mention that \cite{G-H} is a good reference on the topic of regularity of Lorentzian Busemann function in general timelike geodesically complete spacetimes. By the results in Section 3, from now on, every timelike ray (line) with asymptotic direction $\alpha\in(m^{-},m^{+})$ is parametrized by $g$-arc length (if there is no additional assumption).

First, we give the definition of Lorentzian Busemann function for rays.
\begin{defn}\label{def Bu}
Let $(M,g)$ be a non-compact spacetime and $\gamma:\mathbb{R}_{+}\rightarrow M$ be a future complete future-directed timelike ray (parametrized as a geodesic with unit speed). The Lorentzian Busemann function $b_{\gamma}(x): M\rightarrow\mathbb{R}\cup\{\pm\infty\}$ associated to $\gamma$ is defined by
\begin{equation}
b_{\gamma}(x):=\lim_{s\rightarrow\infty}[s-d(x,\gamma(s))].
\end{equation}
By the reverse triangle inequality, the limit in the definition always exists, but it could be infinity!
\end{defn}

We list some of elementary properties of $b_{\gamma}$ without proof, since \cite{G-H} contains all the details we need.
\begin{Pro}\label{Pro Bu}
Let $b_{\gamma}(x):M\rightarrow\mathbb{R}\cup\{\pm\infty\}$ be the Lorentzian Busemann function associated to the ray $\gamma$, then:
\begin{enumerate}
  \item $b_{\gamma}(x)=+\infty$, for $x\in M\setminus I^{-}(\gamma)$;
        $b_{\gamma}(x)<+\infty$, for $x\in I^{-}(\gamma)$.
  \item $b_{\gamma}(x)$ is upper semi-continuous on $I^{-}(\gamma)$.
  \item If $x\in I^{+}(\gamma(0))\cap I^{-}(\gamma)$, $b_{\gamma}(x)$ is finite valued since $b_{\gamma}(x)\geq d(\gamma(0),x)\geq0$.
  \item $b_{\gamma}(x)$ is nondecreasing on causal curves and $b_{\gamma}(y)-b_{\gamma}(x)\geq d(x,y)$ for all $x, y\in I^{+}(\gamma(0))\cap I^{-}(\gamma)$ and $x\leq y$.
\end{enumerate}
\end{Pro}

We shall restrict the function $b_{\gamma}$ on $I^{+}(\gamma(0))\cap I^{-}(\gamma)$, since $b_{\gamma}$ is finite valued on this region by Proposition \ref{Pro Bu}. In our setting, $(M,g)=(\mathbb{R}^{2},g)$ is the Abelian cover of a class A 2-torus.

First, we concern the Lipschitz property of Busemann function associated to timelike rays with asymptotic directions in $(m^{-},m^{+})$.
\begin{The}\label{lip Bu}
Let $K$ be a compact subset of $(m^{-},m^{+})$. Then there exists a uniform number $L(K)>0$, such that for every $\alpha\in K$ and every timelike ray $\gamma$ with asymptotic direction $\alpha$, the Lorentzian Busemann function $b_{\gamma}$ is $L(K)$-Lipschitz (w.r.t. the metric $g_{R}$) on its domain.
\end{The}

\begin{proof}
By the definition of Busemann function, we have
\begin{equation}
b_{\gamma}(x)-b_{\gamma}(y)=\lim_{s\rightarrow\infty}[d(y,\gamma(s))-d(x,\gamma(s))].
\end{equation}
Since every timelike ray with asymptotic direction $\alpha\in(m^{-},m^{+})$ is geodesically complete and then unbounded (defined in Section 1), the conclusion follows from Proposition \ref{A1-2} if the following claim holds.

\textbf{Claim}\label{epsilon estimate}:
\textit{For a fixed point $x\in I^{+}(\gamma(0))$, every $\alpha\in K$ and every timelike pregeodesic $\gamma$ which represents an element of $\mathscr{R}_{\alpha}$, there exist two numbers $\epsilon(K)>0,R(x,\gamma,K)>0$ such that for all $p\in B_{1}(x),T\geq R$,
$$
\gamma(T)-p\in\mathfrak{T}^{\epsilon}.
$$
}
\noindent\textit{Proof of the Claim:}
This claim is just a stronger version of Proposition \ref{asym'} and could be proved by the same way. Only some additional estimates are needed.

First note that
$$
\theta_{K}:=min_{\alpha\in K}\{\text{dist}_{\|\cdot\|}(\alpha,\overline{m}^{\pm})\}>0.
$$

Next, there exists a number $B(g,g_{R})>0$ such that for every $\alpha\in (m^{-},m^{+})$ and every timelike pregeodesic $\gamma$ which could be reparametrized as an element of $\mathscr{R}_{\alpha}$,
\begin{equation}
\text{dist}_{\|.\|}(\gamma(T)-\gamma(0),\overline{\alpha})\leq B(g,g_{R}).
\end{equation}

Finally, by the inequalities
\begin{equation}
\|\gamma(T)-\gamma(0)\|\geq\frac{T}{C(g,g_{R})}-|std(g_{R})|
\end{equation}
(where $C(g,g_{R})$ is the constant arising in Corollary \ref{causal}) and
\begin{eqnarray*}
&&\text{dist}_{\|\cdot\|}(\gamma(T)-p,m^{\pm})\\
&\geq&\text{dist}_{\|\cdot\|}(\gamma(T)-\gamma(0),m^{\pm})-\text{dist}_{\|\cdot\|}(p,\gamma(0)),
\end{eqnarray*}
we obtain that
\begin{eqnarray*}
&&\text{dist}_{\|\cdot\|}(\gamma(T)-p,m^{\pm})\\
&\geq&\theta_{K}\|\gamma(T)-\gamma(0)\|-(\theta_{K}+1)B(g,g_{R})-\text{dist}_{\|\cdot\|}(p,\gamma(0))\\
&\geq&\frac{1}{2}\theta_{K}\|\gamma(T)-\gamma(0)\|
\end{eqnarray*}
if
\begin{equation}
T\geq\frac{2C(g,g_{R})}{\theta_{K}}((\theta_{K}+1)B(g,g_{R})+\text{dist}_{\|\cdot\|}(p,\gamma(0)))+C(g,g_{R})|std(g_{R})|.
\end{equation}
Since $p\in B_{1}(x)$, we know $\text{dist}_{\|\cdot\|}(p,\gamma(0)))\leq A(x,\gamma,g,g_{R})<\infty$.

By choosing $\epsilon=\theta_{K}$ and
$$
R=\frac{2C(g,g_{R})}{\theta_{K}}((\theta_{K}+1)B(g,g_{R})+A(x,\gamma,g,g_{R}))+C(g,g_{R})|std(g_{R})|,
$$
we complete the proof of the claim.\qed

So far, we also complete the proof of Theorem \ref{lip Bu}.
\end{proof}

The following proposition is a direct consequence of the definition of Lorentzian Busemann function (so we omit the proof) and it holds for any non-compact spacetime $(M,g)$.
\begin{Pro}\label{subray bu}
Let $\gamma:\mathbb{R}_{+}\rightarrow(M,g)$ be a timelike ray. For $a>0$, we define a new ray $\gamma_{a}:\mathbb{R}_{+}\rightarrow M$ as $\gamma_{a}(t):=\gamma(t+a)$. Then
$$
b_{\gamma}(x)=b_{\gamma_{a}}(x)+a
$$
on their common domain.
\end{Pro}

The following two propositions reveal the original motivation for defining the Busemann function i.e. to classify which family of rays are parallel by their corresponding Busemann functions. These propositions are proved in a very general case in Riemannian geometry, however, due to the complicated causal structures and loss of regularity of Lorentzian distance function arising from the non-positive definiteness of the Lorentzian metric, they could only be proved in very special cases like ours under the setting of Lorentzian geometry.

\begin{Pro}\label{asym bf}
Let $\zeta$ and $\eta$ be two timelike rays with the same asymptotic direction $\alpha\in(m^{-},m^{+})$. If $\zeta$ and $\eta$ are asymptotic in the future direction, then
\begin{equation}
b_{\zeta}(p)=b_{\eta}(p)-b_{\eta}(\zeta(0))
\end{equation}
on their common domain.
\end{Pro}

\begin{proof}
By the condition that $\lim_{t\rightarrow\infty}d_{R}(\zeta(t),\eta(\mathbb{R}_{+}))\rightarrow0$, we obtain
\begin{equation}\label{1}
\liminf_{t,\bar{t}\rightarrow\infty}d_{R}(\zeta(t),\eta(\bar{t}))=0.
\end{equation}

By the definition of Lorentzian Busemann functions, for any $p,q$ in the domain of $b_{\zeta}$,
\begin{eqnarray*}\label{2}
&&|(b_{\zeta}(p)+b_{\eta}(q))-(b_{\eta}(p)+b_{\zeta}(q))|\\
&\leq&\lim_{t,\bar{t}\rightarrow\infty} [|d(p,\zeta(t))-d(p,\eta(\bar{t}))|+|d(q,\zeta(t))-d(q,\eta(\bar{t}))|].
\end{eqnarray*}

Assume $\zeta\in\mathscr{R}_{\alpha}$. By the argument similar to the Claim in the proof of Theorem \ref{lip Bu}, we have that for such $p$ and $q$, there exists $T>0$ such that
$x-p\in\mathfrak{T}^{\epsilon}$ and $x-q\in\mathfrak{T}^{\epsilon}$ for all $x\in B_{1}(\zeta(t))$ and all $t\geq T$. Here $\epsilon$ is a constant depending only on $\alpha$. By Equation \ref{1}, we could choose $t_{k},\bar{t}_{k}\rightarrow\infty$ such that
\begin{equation}
\lim_{k\rightarrow\infty}d_{R}(\zeta(t_{k}),\eta(\bar{t}_{k}))=0.
\end{equation}
So Proposition \ref{A1-2} implies that
\begin{equation}
\liminf_{t,\bar{t}\rightarrow\infty}[|d(p,\zeta(t))-d(p,\eta(\bar{t}))|+|d(q,\zeta(t))-d(q,\eta(\bar{t}))|]\leq0.
\end{equation}
By monotonicity of $|d(p, \zeta(t))-d(p,\eta(\bar{t})|+|d(q, \zeta(t))-d(q,\eta(\bar{t})|$ w.r.t. $t$ and $\bar{t}$, the $\liminf$ in above formula is indeed a limit and we get that
$$
b_{\zeta}(p)-b_{\eta}(p)\equiv \text{const}.
$$
on their common domain. The remaining part follows by substituting $\eta(0)$ into two sides of the equality.
\end{proof}

We need the concept of co-rays (asymptotes) associated to a ray to show the second proposition and the further regularity of $b_{\zeta}$.

\begin{defn}\label{def co-ray}
Let $\zeta:\mathbb{R}_{+}\rightarrow(\mathbb{R}^{2},g)$ be a timelike ray and $p\in I^{+}(\zeta(0))\cap I^{-}(\zeta)$. For any two sequences $p_{n}\rightarrow p$ and $x_{n}=\zeta(r_{n})$ ($r_{n}\rightarrow\infty$), we have $p_{n}\in I^{-}(x_{n})$ for sufficiently large $n$, $d(p_{n},x_{n})\rightarrow\infty$ (reverse triangle inequality shows that $d(p_{n},x_{n})<\infty$ for large $n$). If $\zeta_{n}:[0,a_{n}]\rightarrow M$ is a maximizing segment connecting $p_{n}$ with $x_{n}$ (the existence of $\zeta_{n}$ is guaranteed by the assumption that $(\mathbb{R}^{2},g)$ is globally hyperbolic) and $\eta:\mathbb{R}_{+}\rightarrow M$ is a limit curve of $\{\zeta_{n}\}$ by \cite[Lemma 2.4]{G-H}, then $\eta$ is called a co-ray associated to the ray $\zeta$ at $p$. If we choose $\zeta_{n}(0)=p$ for all $n$, then the limit curve $\eta$ is called an asymptote, which is a special type of co-ray.
\end{defn}

An immediate application of this definition is the following corollary on the existence of co-ray at every point in the domain of the Lorentzian Busemann function.

\begin{Cor}\label{ex co-ray}
Given a ray $\zeta: {\mathbb{R}}_+ \rightarrow ({\mathbb{R}}^2,g)$ with asymptotic direction $\alpha \in (m^-,m^+)$,  then for every $p\in I^{+}(\zeta(0))$, there exists a future-directed timelike asymptote (co-ray) to the ray $\zeta$, say $\zeta_{p}:\mathbb{R}_{+}\rightarrow(\mathbb{R}^{2},g)$, with $\zeta(0)=p$ and $\zeta_{p}\in\mathscr{R}_{\alpha}$.
\end{Cor}

Roughly speaking, co-rays play a role as the integral curves of the gradient field of the respected Lorentzian Busemann function as the next proposition shows.

\begin{Pro}\label{int curve}
Let $(\mathbb{R}^{2},g)$ be the Abelian cover of a class A Lorentzian 2-torus and $\zeta$ be a timelike ray in $\mathscr{R}_{\alpha}$. Then for any co-ray $\eta:\mathbb{R}_{+}\rightarrow\mathbb{R}^{2}$ associated to $\zeta$ and $0\leq a\leq b$, holds the following equality:
\begin{equation}\label{eq:co-ray}
b_{\zeta}(\eta(b))-b_{\zeta}(\eta(a))=L^{g}(\eta|_{[a,b]}).
\end{equation}
\end{Pro}

\begin{proof}
Let $\{\zeta_{n}\}$ be the maximal segments which converges to $\eta$, as in the definition of co-ray. By Proposition \ref{A1-2} and the definition of $b_{\zeta}$, we have that:
\begin{eqnarray*}
&&b_{\zeta}(\eta(b))-b_{\zeta}(\eta(a))\\
&=&\lim_{n\rightarrow\infty}[d(\eta(a),\zeta(r_{n}))-d(\eta(b),\zeta(r_{n}))]\\
&=&\lim_{n\rightarrow\infty}[d(\zeta_{n}(a),\zeta(r_{n}))-d(\zeta_{n}(b),\zeta(r_{n}))]\\
&=&\lim_{n\rightarrow\infty}[L^{g}(\zeta_{n}|_{[a,a_{n}]})-L^{g}(\zeta_{n}|_{[b,a_{n}]})]\\
&=&\lim_{n\rightarrow\infty}L^{g}(\zeta_{n}|_{[a,b]})\\
&\leq&L^{g}(\eta|_{[a,b]}).
\end{eqnarray*}
Here, $[0,a_{n}]$ denotes the domain of $\zeta_{n}$ as in the definition of co-ray, the second equality follows from the definition of Lorentzian Busemann function, the third equality follows from the Claim in the proof of Theorem \ref{lip Bu} and Proposition \ref{A1-2}, the last inequality follows from Proposition \ref{usc}.
For another direction of the Equality \ref{eq:co-ray}, we have
\begin{eqnarray*}
&&b_{\zeta}(\eta(b))-b_{\zeta}(\eta(a))\\
&\geq&d(\eta(a),\eta(b))\\
&=&L^{g}(\eta|_{[a,b]}),
\end{eqnarray*}
here, the first equality follows from Proposition \ref{Pro Bu}, the last equality follows from the fact that $\eta$ is also a ray.
\end{proof}

Applying the above proposition, we obtain the following corollary.
\begin{Cor}\label{limit int}
Let $\{\zeta_{n}\}$ be a sequence of  co-rays associated to  a timelike ray $\zeta$. If $\zeta_{n}$ converges to a limit curve $\eta$ which is contained in the domain of $b_{\zeta}$, and the Busemann function $b_{\zeta}$ is continuous on the domain, then $\eta$ also satisfies the Equation \ref{eq:co-ray}.
\end{Cor}

\begin{The}\label{sca Bu}
Let $(\mathbb{R}^{2},g)$ be the Abelian cover of a class A Lorentzian 2-torus and $\gamma$ be any timelike ray in $\cup_{\alpha\in K}\mathscr{R}_{\alpha}$, where $K$ is a compact subset of $(m^{-},m^{+})$. Then there exists a positive number $C(K)$ depending only on $K$ such that the Lorentzian Busemann function $b_{\gamma}$ satisfies
\begin{equation}
D^{2}b_{\gamma}(x)\leq C(K)I
\end{equation}
in the sense of upper support function. So $b_{\gamma}$ is locally semi-concave on $I^{+}(\gamma(0))$.
\end{The}

To prove this theorem, we need several technical lemmas which are now well-known to geometers thanks to the efforts of J. Eschenburg, L. Andersson and G. Galloway \cite{A-G-H 1}, \cite{A-G-H 2}.

\begin{Lem}[{\cite[Lemma 2.15]{A-G-H 2}}]\label{sca}
Let $U\subseteq\mathbb{R}^{n}$ be a convex domain and $u:U\rightarrow\mathbb{R}$ be a continuous function. Assume for some constant $c$ and all $p\in U$ that $u$ has a smooth upper support function $u_{p}$ at $p$, i.e. $u_{p}(x)\geq u(x)$ for all $x$ near $p$ with equality holding when $x=p$, such that $D^{2}u_{p}\leq cI$ near $p$. Then $u-\frac{c}{2}\|x\|_{E}^{2}$ is concave in $U$, thus $u$ is semi-concave and twice differentiable almost everywhere in $U$. In this lemma, $\|\cdot\|_{E}$ denotes the Euclidean norm on $\mathbb{R}^{n}$.
\end{Lem}

\begin{Lem}\label{upper spt}
Let $\gamma:\mathbb{R}_{+}\rightarrow(\mathbb{R}^{2},g)$ be a timelike ray in $\mathscr{R}_{\alpha}$ and $b_{\gamma}$ be  the Lorentzian Busemann function associated to $\gamma$. If $\zeta_{p}:\mathbb{R}_{+}\rightarrow\mathbb{R}^{2}$ is a timelike ray emanating from  $p\in I^{+}(\gamma(0))$ that satisfies Equation \ref{eq:co-ray}, then
\begin{equation}
b_{p,\zeta_{p}}(x):=b_{\gamma}(p)+d(p,\zeta_{p}(1))-d(x,\zeta_{p}(1))
\end{equation}
is an upper support function for $b_{\gamma}$ at $p$. Besides, there exists a small neighborhood of $p$ on which $b_{p,\zeta_{p}}$ is smooth and $\nabla b_{p,\zeta_{p}}(p)=-\dot{\zeta}(0)$.
\end{Lem}

\noindent\textit{Proof of Lemma \ref{upper spt}.}
Since $\zeta_{p}$ is a timelike co-ray to $\gamma$ emanating from $p$, then $I^{-}(\zeta_{p}(1))$ is a neighborhood of $p$. So we get from Proposition \ref{Pro Bu} and Equation \ref{eq:co-ray} that for $x\in I^{-}(\zeta_{p}(1))$,
\begin{eqnarray*}
&&b_{\gamma}(p)+d(p,\zeta_{p}(1))-b_{\gamma}(x)\\
&=&b_{\gamma}(\zeta_{p}(0))+L^{g}(\zeta_{p}|_{[0,1]})-b_{\gamma}(x)\\
&=&b_{\gamma}(\zeta_{p}(1))-b_{\gamma}(x)\\
&\geq&d(x,\zeta_{p}(1))
\end{eqnarray*}
with equality at $x=p$. Thus $b_{p,\zeta_{p}}(x)$ is an upper support function of $b_{\gamma}$ at $p$.
The differentiability of $b_{p,\zeta_{p}}(x)$ is obviously equivalent to the differentiability of $d(x,\zeta_{p}(1))$. As $\zeta_{p}$ maximizes the distance between any pair of its points, $\zeta_{p}|_{[0,1]}$ is free of cut points. Thus $d(x,\zeta_{p}(1))$ is smooth near $p$ (\cite[Proposition 9.29]{B-E-E}). Basic knowledge from differential geometry shows that
$$
\nabla b_{p,\zeta_{p}}(p)=-\nabla d(x,\zeta_{p}(1))=-\dot{\zeta_{p}}(0).
$$
\qed

\noindent\textit{Proof of Theorem \ref{sca Bu}.}
By Lemma \ref{sca}, we only need to prove that there exists a constant $C(K)>0$ depending only on $K$ such that $D^{2}b_{p,\zeta_{p}}(x)\leq C(K)I$ near $p$. This is equivalent to prove that $D^{2}d(x,\zeta_{p}(1))\geq-C(K)I$ near $p$. It is a standard argument using comparison theorem to give an estimation of the Hessian (defined in terms of the Levi-Civita connection w.r.t. $g$) of $d(x,\zeta_{p}(1))$ in terms of the lower and upper bound of timelike sectional curvature of planes containing $\dot{\zeta_{p}}(t)$ and the length of maximal segment connecting $p$ with $\zeta_{p}(1)$ (a standard reference is \cite{A-H}). Here, we only need to consider the bounds of sectional curvature since $\zeta_{p}|_{[0,1]}$ is maximal and the length is always $1$.

Note that the dimension of the spacetime we consider is of two, so the sectional curvature reduces to the Gauss curvature. Since the metric $g$ is a periodic lift of $(\mathbb{T}^{2},g)$, the Gauss curvature is uniformly bounded. According to this, we get the estimate for the Hessian of $d(x,\zeta_{p}(1))$ by using the method in \cite[Proposition 3.1]{A-G-H 1}. By the estimates on Hessian of $d(x,\zeta_{p}(1))$ and Lemma \ref{sca}, Theorem \ref{sca Bu} holds.\qed

Now we could give the definition of Lorentzian Busemann function for a line which will be used in the proof of our main results. For any timelike line $\gamma:\mathbb{R}\rightarrow(\mathbb{R}^{2},g)$ in $\mathscr{M}_{\alpha}$, there is a sequence of rays $\gamma_{k}:\mathbb{R}_{+}\rightarrow(\mathbb{R}^{2},g),k\in\mathbb{N}$ such that $\gamma_{k}(0)=\gamma(-k)$. We denote the Lorentzian Busemann function associated to $\gamma_{k}$ by $b_{k}:I^{+}(\gamma(-k))\cap I^{-}(\gamma)\rightarrow\mathbb{R}$. By Remark \ref{dom'}, the domain of $b_{k}$ is clearly $I^{+}(\gamma(-k))$, thus the domain of $b_{k}$ is strictly contained in the domain of $b_{k+1}$. By Proposition \ref{subray bu}, $b_{k}$ and $b_{k+1}$ differ by a constant on their common domain.

Let $C^{0}(M,\mathbb{R})$ be the set of all continuous functions on $M$. One could define an equivalence relation $\sim$ among functions in $C^{0}(M,\mathbb{R})$ as the following:
$$
f\sim g\text{ if and only if }f-g\equiv const. \text{ on }M.
$$
Denote $C^{0}(M,\mathbb{R})/\mathbb{R}$ to be the equivalence classes under $\sim$, it is also a vector space over $\mathbb{R}$. For $F\in C^{0}(M,\mathbb{R})/\mathbb{R}$, we call a function $f\in C^{0}(M,\mathbb{R})$ a representative of $F$ if and only if $f\in F$.

\begin{defn}\label{line Bu}
For every timelike line $\gamma\in\mathscr{M}_{\alpha}$ with $\alpha\in(m^{-},m^{+})$, there is a unique element $b_{\gamma}$ in $C^{0}(\mathbb{R}^{2},\mathbb{R})/\mathbb{R}$ such that for any representative $f$ of $b_{\gamma}$ and any $k\in\mathbb{N}$,
\begin{equation}\label{line Bu'}
f(x)-b_{k}(x)\equiv const._{k},x\in I^{+}(\gamma(-k)),
\end{equation}
where $const._{k}$ denotes a constant depending only on $k$. We call $b_{\gamma}$ or its representatives the Lorentzian Busemann function associated to the line $\gamma$.
\end{defn}

\begin{Rem}
By  Proposition \ref{dom} and Proposition \ref{A1-2}, $b_{\gamma}$ exists and its representatives are Lipschitz  on $\mathbb{R}^{2}$. Thus the representatives are differentiable almost everywhere. In the remaining context, we do not distinguish $b_{\gamma}$ and its representatives. When a line $\gamma\in\mathscr{M}_{\alpha}$ is given, $b_{\gamma}$ represents either the unique function $f$ in $C^{0}(M,\mathbb{R})$ satisfies the Equation \ref{line Bu'} with value $0$ at the point $\gamma(0)$ or the element in $C^{0}(\mathbb{R}^{2},\mathbb{R})/\mathbb{R}$ that $f$ belongs to. We shall switch these two meanings in several cases if there is no confusion.
\end{Rem}

To complete this section, we note that we could get the relationship between Lorentzian Busemann function for rays in $\mathscr{R}_{\alpha}$ and Lorentzian Busemann function for lines in $\mathscr{M}_{\alpha}$ defined in Definition \ref{line Bu} by Theorem \ref{irrational}, Theorem \ref{rational} and Proposition \ref{asym bf}.

\begin{The}\label{re bu-rl}
Let $\zeta:\mathbb{R}_{+}\rightarrow(\mathbb{R}^{2},g)$ be a timelike ray in $\mathscr{R}_{\alpha}$ which is not a subray of any line in $\mathscr{M}_{\alpha}$. If $\alpha$ is irrational, and $\zeta$ is asymptotic to both $\underline{\zeta}$ and $\overline{\zeta}$ in the future direction, then both
$$
b_{\underline{\zeta}}=b_{\zeta}+b_{\underline{\zeta}}(\zeta(0))
$$
and
$$
b_{\overline{\zeta}}=b_{\zeta}+b_{\overline{\zeta}}(\zeta(0))
$$
hold on $I^{+}(\zeta(0))$.

If $\alpha$ is rational, $\zeta$ is asymptotic to either $\underline{\zeta}$ or $\overline{\zeta}$ in the future direction, then
\begin{equation}
b_{\gamma}=b_{\zeta}+b_{\gamma}(\zeta(0))
\end{equation}
holds on $I^{+}(\zeta(0))$ if $\gamma\in\{\underline{\zeta},\overline{\zeta}\}$ and $\zeta$ is asymptotic to $\gamma$ in the future direction.
\end{The}

\begin{Rem}\label{rem line bu}
After all, by Theorem \ref{re bu-rl}, we could reduce the study of Lorentzian Busemann functions for rays to the study of Lorentzian Busemann function for lines in $\mathscr{M}_{\alpha}$. Lorentzian Busemann function for lines in $\mathscr{M}_{\alpha}$ are nice objects for investigation since they are globally defined and obviously share same regularities as Lorentzian Busemann functions for rays that have been proved in Theorem \ref{lip Bu} and Theorem \ref{sca Bu}.
\end{Rem}

\section{Proof of the main results}
In this section, we shall study some further properties of Lorentzian Busemann functions for timelike lines with asymptotic directions in $(m^{-},m^{+})$ and then prove Theorem \ref{main theorem 1} and Theorem \ref{main theorem 2}.

Let us illustrate the relations between our main results (i.e. Theorem \ref{main theorem 1}, Theorem \ref{main theorem 2}) and the theorems in this section. Based on Theorem \ref{gls} and Proposition \ref{wcf}, the construction of $\omega_{\alpha}$ and $b_{\alpha}$ mentioned in Theorem \ref{main theorem 1} is given by Definition \ref{irrational global vis} if $\alpha$ is irrational and is given by Theorem \ref{rational global vis} if $\alpha$ is rational. Based on the construction of $b_{\alpha}$, the first item of Theorem \ref{main theorem 1} is proved by Theorem \ref{vis},  Proposition \ref{wcf}, Remark \ref{wcf'} and the second one is proved by Remark \ref{rem line bu}, the third item of Theorem \ref{main theorem 1} is proved by Theorem \ref{weakkam}, Corollary \ref{weakkam1} and Theorem \ref{weakkam2}. The last item of Theorem \ref{main theorem 1} follows from Theorem \ref{gls} and the construction of $b_{\alpha}$. For Theorem \ref{main theorem 2}, the first item is proved by Theorems \ref{gls} and \ref{unique} directly and the second one is proved by Proposition \ref{wcf} and Theorem \ref{non-diff}.

By Theorem \ref{irrational}, Theorem \ref{rational} and Proposition \ref{asym bf}, we only need to consider two kinds of Lorentzian Busemann functions $b_{\gamma}$, that are associated to  $\gamma\in\mathscr{M}_{\alpha}^{rec}$ when $\alpha$ is irrational and $\gamma\in\mathscr{M}_{\alpha}^{per}$ when $\alpha$ is rational. Like before, we denote the Deck transformations by $T_{(i,j)}$.

The following definitions are useful for us.
\begin{defn}\label{limiting gradient}
Let $u:M\rightarrow\mathbb{R}$ be a locally Lipschitz function defined on the Lorentzian manifold $(M,g)$, then a vector $V\in T_{q}M$ is called a limiting gradient if there exists a sequence $\{q_{k}\}\subset M\setminus\{q\}$ with $\lim_{k\rightarrow\infty}q_{k}=q$ such that $u$ is differentiable at $q_{k}$ for each $k\in\mathbb{N}$, and  $\lim_{k\rightarrow\infty}\nabla u(q_{k})=V$. Here, the first limit is taken in the sense of the manifold topology on $M$;  the second limit is taken in the sense of any fixed chart that contains $q$. Since the first limit is taken, we know that when $k$ is sufficiently large, $q_{k}$ goes into that chart. The second limit does not depend on the choice of chart.
\end{defn}
We denote $\nabla^{\ast}u(q)$ to be the set of all limiting gradients of $u$ at $q$.  For a set $A$  in a vector space, the convex hull of $A$, $coA$, is the smallest convex set containing $A$.

By the knowledge of  convex analysis, we know the following relationship between these two sets.

\begin{Lem}\label{sca limit gradient}
If $u$ is a locally semiconcave function on manifold $M$, then it is locally Lipschitz (under any reasonable metric), and $\nabla^{+}u(q)$ is non-empty for any $q\in M$. In this case, $\nabla^{+}u(q)=co\nabla^{\ast}u(q)\subset T_{q}M$.
\end{Lem}

For a proof of this lemma, see \cite[Theorem 3.3.6]{C-S}, where limiting gradient is called reachable gradient.

\begin{defn}\label{gradient line}
Let $\gamma\in\mathscr{M}_{\alpha}$ with $\alpha\in(m^{-},m^{+})$ be a timelike line and $b_{\gamma}:\mathbb{R}^{2}\rightarrow\mathbb{R}$ be the Lorentzian Busemann function associated to $\gamma$. If a timelike ray $\zeta:\mathbb{R}_{+}\rightarrow(\mathbb{R}^{2},g)$ satisfies
\begin{equation}\label{eq:weakkam line}
b_{\gamma}(\zeta(t))-b_{\gamma}(\zeta(0))=t,
\end{equation}
we say $\zeta$ is a gradient line for $b_{\gamma}$. We denote the set of all gradient lines for $b_{\gamma}$ by $\mathfrak{C}_{\gamma}$.
\end{defn}

\begin{Rem}\label{grd}
We remark that from Proposition \ref{Pro Bu}, one easily deduces that any timelike curve $\zeta:\mathbb{R}_{+}\rightarrow(\mathbb{R}^{2},g)$ satisfying Equation \ref{eq:weakkam line} is a timelike ray.
\end{Rem}

As an application of Definition \ref{gradient line}, we have the following lemma.
\begin{Lem}\label{3}
Let $\gamma\in\mathscr{M}_{\alpha},\alpha\in(m^{-},m^{+})$. If $\zeta,\eta$ are two timelike rays in $\mathfrak{C}_{\gamma}$ and $\eta(0)=\zeta(a)$ for some $a>0$, then $\eta(t)=\zeta(t+a)$ for $t\in\mathbb{R}_{+}$.
\end{Lem}

\begin{proof}
Set $p=\eta(0)=\zeta(a),x=\zeta(0),y=\eta(a)$. If the conclusion does not hold, then $\dot{\eta}(0)\neq\dot{\zeta}(a)$ since $\eta$ and $\zeta$ are all $g$-geodesics. So the conjunction curve $\zeta|_{[0,a]}\ast\eta|_{[0,a]}$ has a corner at $p$, and \begin{eqnarray*}\label{ieq:11}
&&d(x,p)+d(p,y)\\
&=&d(\zeta(0),\zeta(a))+d(\eta(0),\eta(a))\\
&=&L^{g}(\zeta|_{[0,a]}\ast\eta|_{[0,a]})\\
&<&d(\zeta(0),\eta(a))\\
&=&d(x,y).
\end{eqnarray*}
On the other hand, since $\eta$ and $\zeta$ are in $\mathfrak{C}_{\gamma}$, we get
\begin{eqnarray*}
&&b_{\gamma}(p)-b_{\gamma}(x)\\
&=&b_{\gamma}(\zeta(a))-b_{\gamma}(\zeta(0))\\
&=&d(\zeta(0),\zeta(a))\\
&=&d(x,p).
\end{eqnarray*}
Similarly, $b_{\gamma}(y)-b_{\gamma}(p)=d(p,y)$.

So $b_{\gamma}(y)-b_{\gamma}(x)=d(x,p)+d(p,y)<d(x,y)$. This contradicts the last item in Proposition \ref{Pro Bu}.
\end{proof}

The first theorem asserts that every Lorentzian Busemann function for a line with asympototic direction in $(m^-,m^+)$ is a global viscosity solution of the eikonal equation $(*)$ on $(\mathbb{R}^{2},g)$.
\begin{The}\label{vis}
If $\alpha\in(m^{-},m^{+})$ and $\gamma\in\mathscr{M}_{\alpha}$, then $b_{\gamma}:\mathbb{R}^{2}\rightarrow\mathbb{R}$ is a viscosity solution to the eikonal equation $(*)$.
\end{The}

\begin{proof}
By Theorem \ref{lip Bu} and Remark \ref{rem line bu}, $b_{\gamma}$ is Lipschitz on $\mathbb{R}^{2}$ w.r.t. $d_{R}$, thus is differentiable almost everywhere. By Theorem \ref{sca Bu} and Remark \ref{rem line bu}, $b_{\gamma}$ is locally semi-concave with linear modulus on $\mathbb{R}^{2}$. So for any $q$, the supergradients of $b_{\gamma}$ at $q$ form a nonempty, compact and convex set $\nabla^{+}b_{\gamma}(q)$ in $T_{q}\mathbb{R}^{2}$. Let $\nabla^{-}b_{\gamma}(q)$ denote the set of subgradients of $b_{\gamma}$ at $q$. If both $\nabla^{-}b_{\gamma}(q)$ and $\nabla^{+}b_{\gamma}(q)$ are non-empty, then $b_{\gamma}$ is differentiable at $q$ and $\nabla^{-}b_{\gamma}(q)=\nabla^{+}b_{\gamma}(q)=\{\nabla b_{\gamma}(q)\}$ in this case.

At a differentiable point $q$ of $b_{\gamma}$, we choose an arbitrary smooth future directed causal curve $\eta:[0,\delta)\rightarrow\mathbb{R}^{2}$ with $\eta(0)=q$. By Proposition \ref{Pro Bu}, we obtain
\begin{eqnarray*}
&&g(\nabla b_{\gamma}(q),\dot{\eta}(0))\\
&=&db_{\gamma}(q)(\dot{\eta}(0))\\
&=&\lim_{t\rightarrow0^{+}}\frac{b_{\gamma}(\eta(t))-b_{\gamma}(q)}{t}\\
&\geq&\limsup_{t\rightarrow0^{+}}\frac{d(\eta(0),\eta(t))}{t}\\
&\geq&\lim_{t\rightarrow0^{+}}\frac{1}{t}\int_{0}^{t}\sqrt{-g(\dot{\eta}(s),\dot{\eta}(s))}ds\\
&=&\sqrt{-g(\dot{\eta}(0),\dot{\eta}(0))}.
\end{eqnarray*}
So we get that for any differentiable point $q$ of $b_{\gamma}$ and any future directed causal vector $\dot{\eta}(0)\in T_{q}\mathbb{R}^{2}$,
\begin{equation}\label{ieq:6}
g(\nabla b_{\gamma}(q),\dot{\eta}(0))\geq\sqrt{-g(\dot{\eta}(0),\dot{\eta}(0))}.
\end{equation}
By \cite[Lemma 2.3, 2.4]{C-J} and Inequality \ref{ieq:6}, $\nabla b_{\gamma}(q)$ is a past directed timelike vector and
\begin{equation}\label{eq:Bu}
g(\nabla b_{\gamma}(q),\nabla b_{\gamma}(q))=-1
\end{equation}
at every differentiable point $q$ of $b_{\gamma}$. This proves that $b_{\gamma}$ is a viscosity supersolution of the eikonal equation $(*)$, since $\nabla^{-}b_{\gamma}(q)$ is empty at any nondifferentiable point $q$ of $b_{\gamma}$.

Since $b_{\gamma}$ is locally semi-concave, Definition \ref{limiting gradient} and Lemma \ref{sca limit gradient} imply that $\nabla^{\ast}b_{\gamma}(q)$ is never empty and $\nabla^{+}b_{\gamma}(q)=co\nabla^{\ast}b_{\gamma}(q)$. By the smoothness of $g$ and Equation \ref{eq:Bu}, we obtain
$$
\nabla^{\ast}b_{\gamma}(q)\subseteq\{V\in T_{q}\mathbb{R}^{2}|V\text{ is past directed,  } g(V,V)=-1\}.
$$
This proves $b_{\gamma}$ is a viscosity subsolution of the eikonal equation $(*)$, since
\begin{eqnarray*}
&&\nabla^{+}b_{\gamma}(q)\\
&=&co\nabla^{\ast}b_{\gamma}(q)\\
&\subseteq&co\{V\in T_{q}\mathbb{R}^{2}|V\text{ is past directed,  } g(V,V)=-1\}\\
&=&\{V\in T_{q}\mathbb{R}^{2}|V\text{ is past directed,  } g(V,V)\leq-1\}.
\end{eqnarray*}
\end{proof}

The second theorem shows that timelike rays in $\mathfrak{C}_{\gamma}$ which emanate from $q$ have a one to one correspondence with the vectors in $\nabla^{\ast}b_{\gamma}(q)$.

\begin{The}\label{weakkam}
If $\alpha\in(m^{-},m^{+})$ and $\gamma\in\mathscr{M}_{\alpha}$, then for every $q\in\mathbb{R}^{2}$ and every $V\in\nabla^{\ast}b_{\gamma}(q)$, there is a unique timelike ray in $\mathscr{R}_{\alpha}$, namely $\gamma_{q,V}:\mathbb{R}_{+}\rightarrow(\mathbb{R}^{2},g)$, satisfies Equation \ref{eq:weakkam line} and $\dot{\gamma_{q}}(0)=-V$.
\end{The}

\begin{proof}
The uniqueness of such a timelike ray is guaranteed by the uniqueness of the solution that satisfies the second order geodesic equation (w.r.t. the Lorentzian metric $g$) with the initial condition $\gamma_{q,V}(0)=q;\dot{\gamma}_{q,V}(0)=-V$. So we only consider the existence of such a ray.

By Corollary \ref{ex co-ray}, there is a co-ray to $\gamma$ emanating from $q$ and this co-ray is in $\mathscr{R}_{\alpha}$. By Proposition \ref{int curve}, every such co-ray satisfies Equation \ref{eq:weakkam line}.

If $b_{\gamma}$ is differentiable at $q$, then $\nabla^{+}b_{\gamma}(q)=\nabla^{\ast}b_{\gamma}(q)$ is a singleton. In this case, if we denote the co-ray emanating from $q$ to $\gamma$ by $\gamma_{q}$, there is a $C^{1}$ upper support function of $b_{\gamma}$, namely $b_{q,\gamma_{q}}$, in a neighborhood of $q$ by Lemma \ref{upper spt}. By Definition \ref{generalized gradients}, one concludes that $\nabla b_{q,\gamma_{q}}(q)=-\dot{\gamma}_{q}(0)=\nabla b_{\gamma}(q)$, which is the unique vector in $\nabla^{+}b_{\gamma}(q)=\nabla^{\ast}b_{\gamma}(q)$.

By what we have proved, $\nabla^{\ast}b_{\gamma}(q)$ is nonempty everywhere. Assume $V\in\nabla^{\ast}b_{\gamma}(q)$, then there is a sequence of differentiable points of $b_{\gamma}$, say $\{q_{k}\}$, converges to $q$ and satisfies
\begin{equation}
\lim_{k\rightarrow\infty}\nabla b_{\gamma}(q_{k})=V.
\end{equation}
Thus there is correspondly a sequence of curves $\{\gamma_{q_{k}}\}\subseteq\mathscr{R}_{\alpha}$ that satisfies Equation \ref{eq:weakkam line} and $\dot{\gamma}_{q_{k}}(0)=-\nabla b_{\gamma}(q_{k})$. Since $\gamma\in\mathscr{M}_{\alpha}$, there is a Lipschitz constant $L(\alpha)>0$ for $b_{\gamma}$. So the $g_{R}$-norms of $\dot{\gamma}_{q_{k}}(t)$ have a uniform upper bound $L(\alpha)$ for any $t\in\mathbb{R}_{+}$ and $k\in\mathbb{N}$. Using this fact, we obtain that, as a family of $C^{1}$ maps from $\mathbb{R}_{+}$ to T$\mathbb{R}^{2}$, $(\gamma_{q_{k}}(t),\dot{\gamma}_{q_{k}}(t))$ is equi-Lipschitz in $k$ w.r.t. the induced metric on T$\mathbb{R}^{2}$ by $g_{R}$ (since every $\gamma_{q_{k}}$ satisfies the second order geodesic equation w.r.t. $g$). Thus by Ascoli-Arzela theorem and a diagonal trick, there is a $g$-geodesic $\gamma_{q,V}$ emanating from $q$ such that some subsequence of $\gamma_{q_{k}}$ converges to it in the $C^{1}$-topology. By Corollary \ref{limit int} and the definition of $\gamma_{q,V}$, we obtain that $\gamma_{q,V}$ satisfies Equation \ref{eq:weakkam line} and $\gamma_{q,V}(0)=q$, $\dot{\gamma}_{q,V}(0)=-V\in\nabla^{\ast}b_{\gamma}(q)$. Finally, since $\gamma_{q_{k}}\in\mathscr{R}_{\alpha}$, by Corollary \ref{ray osc}, $\gamma_{q,V}$ is also in $\mathscr{R}_{\alpha}$.
\end{proof}

We have the following corollary of Theorem \ref{weakkam}.
\begin{Cor}\label{weakkam1}
If $\alpha\in(m^{-},m^{+})$ and $\gamma\in\mathscr{M}_{\alpha}$, then $b_{\gamma}$ is differentiable at $q\in \mathbb{R}^{2}$ if and only if there exists a unique timelike ray $\gamma_{q}$ (emanating from $q$) in $\mathscr{R}_{\alpha}\cap\mathfrak{C}_{\gamma}$ and satisfies $\dot{\gamma_{q}}(0)=-\nabla b_{\gamma}(q)$.
\end{Cor}

\begin{Rem}\label{4}
From the proof of Theorem \ref{weakkam}, one easily deduce that every $\gamma_{q,V}$ is a co-ray associated to $\gamma$. Because of the assumption that the dimension of the spacetime is of two, the asymptote emanating from a fixed point $q$ is unique. Thus if $b_{\gamma}$ is differentiable at $q$, the only co-ray emanating from $q$ coincides with the asymptote; if $b_{\gamma}$ is not differentiable at $q$, then there is only one $V\in\nabla^{*}b_{\gamma}(q)$ such that $\gamma_{q,V}$ is the asymptote emanating from $q$.
\end{Rem}

\begin{The}\label{weakkam2}
Assume that $\alpha\in(m^{-},m^{+})$ and $\gamma\in\mathscr{M}_{\alpha}$. For any $\gamma_{q}\in\mathfrak{C}_{\gamma}$ emanating from $q$, $b_{\gamma}$ is differentiable at $\gamma_{q}(t)$ for any $t>0$ and $\dot{\gamma_{q}}(t)=-\nabla b_{\gamma}(\gamma_{q}(t))$.
\end{The}

\begin{proof}
Suppose $b_{\gamma}$ is not differentiable at $\gamma_{q}(a)$ for some $a>0$, then $\nabla^{\ast}b_{\gamma}(\gamma_{q}(a))$ contains at least two elements. By Theorem \ref{weakkam}, there is a timelike ray $\eta$ in $\mathfrak{C}_{\gamma}$ that emanates from $\gamma_{q}(a)$ and $\dot{\eta}(0)\neq\dot{\gamma_{q}}(a)$. Replacing $\zeta$ by $\gamma_{q}$ in Lemma \ref{3}, we obtain a contradiction.
\end{proof}

The following theorem describes the set of gradient lines $\mathfrak{C}_{\gamma}$. We need a lemma which is useful in the proof of our theorem.

\begin{Lem}\label{asym grad}
Let $\zeta$ and $\eta$ be two timelike rays in $\mathscr{R}_{\alpha}$, both of them emanate from $q$. If $\zeta$ and $\eta$ are asymptotic in the future direction and $\zeta\in\mathfrak{C}_{\gamma}$, then $\eta\in\mathfrak{C}_{\gamma}$.
\end{Lem}

\begin{proof}
Set $f(x):=b_{\gamma}(x)-b_{\gamma}(q)-t$. Then $\eta$ is in $\mathfrak{C}_{\gamma}$ if and only if $f(\eta(t))\equiv0$. By the last item in Proposition \ref{Pro Bu}, $f(\eta(t))$ is non-negative and non-decreasing in $t$. The fact that $\zeta$ is in $\mathfrak{C}_{\gamma}$ implies that  $f(\zeta(t))\equiv0$ on $\mathbb{R}_{+}$. Since $\zeta$ and $\eta$ are asymptotic in the future direction,  we could choose $a_{k},a^{\prime}_{k}\rightarrow\infty$ such that
\begin{equation}\label{eq:limit}
\begin{split}
d_{R}(\zeta(a_{k}),\eta(a^{\prime}_{k}))\rightarrow0,\\
|a_{k}-a^{\prime}_{k}|\rightarrow0.
\end{split}
\end{equation}
Here, the second limit follows from the first one by using Proposition \ref{A1-2} and the fact that $\zeta,\eta$ are in $\mathscr{R}_{\alpha}$.

By Theorem \ref{lip Bu}, we denote the Lipschitz constant of $f$ by $L(\alpha)$. Then we use Equation \ref{eq:limit} to obtain
\begin{eqnarray*}
&&f(\eta(a^{\prime}_{k}))\\
&=&|f(\eta(a^{\prime}_{k}))-f(\zeta(a_{k}))|\\
&\leq&|a^{\prime}_{k}-a_{k}|+|b_{\gamma}(\eta(a^{\prime}_{k}))-b_{\gamma}(\zeta(a_{k}))|\\
&\leq&|a^{\prime}_{k}-a_{k}|+L(\alpha)d_{R}(\eta(a^{\prime}_{k}),\zeta(a_{k}))\\
&\rightarrow&0.
\end{eqnarray*}
By the non-decreasing property of $f(\eta(t))$, we obtain $f(\zeta(t))\equiv0$, so $\zeta\in\mathfrak{C}_{\gamma}$.
\end{proof}

\begin{The}\label{gls}
If $\alpha\in(m^{-},m^{+})$ is irrational, then $\mathfrak{C}_{\gamma}=\mathscr{R}_{\alpha}$. If $\alpha\in(m^{-},m^{+})$ is rational and $\gamma\in\mathscr{M}_{\alpha}^{per}$, then
$\mathfrak{C}_{\gamma}=\mathscr{R}_{\alpha}^{per}\cup\{\zeta\in\mathscr{R}_{\alpha}^{+}|\zeta\leq\gamma\}\cup\{\zeta\in\mathscr{R}_{\alpha}^{-}|\zeta\geq\gamma\}$.
\end{The}

\begin{proof}
First, we prove $\mathfrak{C}_{\gamma}\subseteq\mathscr{R}_{\alpha}$. Choose any $\zeta\in\mathfrak{C}_{\gamma}$ emanating from $q$. By Theorem \ref{weakkam2}, $\zeta$ is differentiable at $\zeta(t)$ and $-\dot{\zeta}(t)\in\nabla b_{\gamma}(\zeta(t))$ for any $t>0$. So $-\dot{\zeta}(0)\in\nabla^{\ast}b_{\gamma}(q)$, $\zeta$ must coincide with the timelike ray $\gamma_{q,-\dot{\zeta}(0)}$ that we obtained by Theorem \ref{weakkam}. Thus $\zeta\in\mathscr{R}_{\alpha}$.

If $\alpha$ is irrational, we know that from every point $q\in\mathbb{R}^{2}$, there is at least one co-ray to $\gamma$, say $\gamma_{q}$, emanating from $q$ and $\gamma_{q}\in\mathfrak{C}_{\gamma}$ by Proposition \ref{int curve}. By Theorem \ref{irrational}, any $\zeta\in\mathscr{R}_{\alpha}$ emanating from $q$ either coincides with $\gamma_{q}$ or is asymptotic to it in the future direction. So by Lemma \ref{asym grad}, $\zeta\in\mathfrak{C}_{\gamma}$. This proves $\mathfrak{C}_{\gamma}=\mathscr{R}_{\alpha}$ when $\alpha$ is irrational.

Assume $\alpha$ is rational and $\gamma\in\mathscr{M}_{\alpha}^{per}$. Since $\mathscr{M}^{per}_{\alpha}$ is closed (as a subset of the plane) and totally ordered, then if $q$ is not in the image of any periodic line in $\mathscr{M}_{\alpha}^{per}$, it falls in a gap bounded by two neighboring periodic lines $\xi>\eta$ in $\mathscr{M}_{\alpha}^{per}$. Now there are two cases we shall consider, namely $q>\gamma$ and $q<\gamma$. Since the proof of these two cases are completely similar, we assume $q>\gamma$, thus $q>\eta\geq\gamma$.

Let us recall the construction of asymptote from Definition \ref{def co-ray}. One begins with a point $q\in\mathbb{R}^{2}$ and a sequence $r_{n}\rightarrow\infty (n\in\mathbb{N})$, then it follows from Remark \ref{dom'} that for sufficiently large number $r_{n}$, we have $q\ll\gamma(r_{n})$. Connecting $p$ with $\gamma(r_{n})$ by a timelike maximal segment $\zeta_{n}$ (now parametrize them by $g$-arc length), one could easily see $\{\zeta_{n}\}_{n\in\mathbb{N}}$, by passing to a subsequence if necessary, will converge to a timelike ray $\zeta_{0}\in\mathscr{R}_{\alpha}$ emanating from $p$ w.r.t. the $C^{0}$ topology on $C^{0}(\mathbb{R}_{+},\mathbb{R}^{2})$.

If $\eta=\gamma$, then from the proof of Theorem \ref{rational}, we see that the asymptote $\zeta_{0}$ defined above is in $\mathscr{R}_{\alpha}^{-}$. Since every asymptote is in $\mathfrak{C}_{\gamma}$, by Lemma \ref{asym grad}, every timelike ray $\zeta\in\mathscr{R}_{\alpha}^{-}$ emanating from $q$ must be in $\mathfrak{C}_{\gamma}$. On the other hand, by Lemma \ref{3}, any two rays in $\mathfrak{C}_{\gamma}$ cannot intersect each other at any interior point, so any ray in $\mathscr{R}_{\alpha}^{+}$ cannot be in $\mathfrak{C}_{\gamma}$ since it will intersect some line in $\mathscr{M}_{\alpha}^{-}$. If $\eta>\gamma$, Jordan curve theorem and Morse's lemma in Section 4 imply that every $\zeta_{n}$ intersects $\eta$ at a unique point $p_{n}$. If the set $\{p_{n}\}$ has a limit point $p$, then one apply the convergence process to the corresponding subsequence of $\{\zeta_{n}\}$ and obtain a timelike ray $\zeta_{0}\in\mathscr{R}_{\alpha}$ which intersect $\eta$ at $p$. This contradicts Theorem \ref{rational} and thus $p_{n}$ goes to infinity. Denote the timelike maximal segment connecting $q$ with $p_{n}$ by $\xi_{n}$, then $\xi_{n}$ and $\zeta_{n}$ has the same limit curve $\zeta_{0}$. So $\zeta_{0}\in\mathscr{R}_{\alpha}^{-}$ as in the last paragragh. The remaining proof is the same as the case $\eta=\gamma$.

The above discussions prove that if $q>\gamma$ falls in a gap bounded by two neighboring periodic lines, then any timelike ray $\zeta$ emanating from $q$ is in $\mathfrak{C}_{\gamma}$ if and only if $\zeta\in\mathscr{R}_{\alpha}^{-}$. Similarly, if $q<\gamma$ falls in a gap bounded by two neighboring periodic lines, then any timelike ray $\zeta$ emanating from $q$ is in $\mathfrak{C}_{\gamma}$ if and only if $\zeta\in\mathscr{R}_{\alpha}^{+}$.

If $q$ belongs to some periodic line $\gamma_{0}$, then the subray of $\gamma_{0}$ emanating from $q$ is the only co-ray to $\gamma$ emanating from $q$. By Remark \ref{4}, every timelike ray in $\mathfrak{C}_{\gamma}$ coincides with a co-ray to $\gamma$ emanating from the same point, so the subray of $\gamma_{0}$ emanating from $q$ is the only timelike ray in $\mathfrak{C}_{\gamma}$ that emanates from $q$.
\end{proof}

\begin{Pro}\label{wcf}
We list two well known but important facts:
\begin{itemize}
  \item If two Lipschitz functions defined on a common open set $\Omega$ have the same gradients (induced by a Riemannian or Lorentzian metric) on their common differentiable points, then they must differ by a constant.
  \item For a Lipschitz function $f:\mathbb{R}^{2}\rightarrow\mathbb{R}$, if for any $(i,j)\in\mathbb{Z}^{2}, f\circ T_{(i,j)}-f\equiv const.$, then $f$ could be written as a sum of a linear function and a lift of a Lipschitz function defined on $\mathbb{R}^{2}/\mathbb{Z}^{2}$.
\end{itemize}
\end{Pro}

\begin{Rem}\label{wcf'}
The differential of a function $f$ mentioned in the second item of Proposition \ref{wcf} exists almost everywhere w.r.t. the Lebesgue measure on $\mathbb{R}^{2}$ and descends to a weakly closed 1-form on $\mathbb{T}^{2}$.
\end{Rem}

If $\alpha$ is irrational, by Theorem \ref{gls} and Proposition \ref{wcf}, for any two timelike lines $\gamma,\xi$ belong to $\mathscr{M}_{\alpha}$, $b_{\gamma}$ and $b_{\xi}$ could be seen as the same element of $C^{0}(\mathbb{R}^{2},\mathbb{R})/\mathbb{R}$. This fact stimulates the following definition.

\begin{defn}\label{irrational global vis}
Let $(\mathbb{R}^{2},g)$ be the Abelian cover of a class A Lorentzian 2-torus, and $\alpha\in(m^{-},m^{+})$ be an irrational asymptotic direction. Then there exists a unique element $b_{\alpha}\in C^{0}(\mathbb{R}^{2},\mathbb{R})/\mathbb{R}$ such that for every $\gamma\in\mathscr{M}_{\alpha}$, the equality $b_{\gamma}=b_{\alpha}$
holds as elements of $C^{0}(\mathbb{R}^{2},\mathbb{R})/\mathbb{R}$. We call this $b_{\alpha}$ or any of its representative the global Lorentzian Busemann function for the irrational asymptotic direction $\alpha$.
\end{defn}

It is obvious that if $\alpha$ is irrational, $\mathscr{R}_{\alpha}$ is $\mathbb{Z}^{2}$-invariant, so by Theorems \ref{rem line bu}, \ref{vis}, \ref{gls} and Corollaries \ref{weakkam1}, \ref{weakkam2}, we conclude that for every irrational $\alpha$, any representative of $b_{\alpha}$ satisfies all the requirements in Theorem \ref{main theorem 1} and the first item of Theorem \ref{main theorem 2}.

\begin{The}\label{unique}
Let $(\mathbb{R}^{2},g)$ be the Abelian cover of a class A Lorentzian 2-torus, $\alpha\in(m^{-},m^{+})$ an irrational asymptotic direction. If $u:\mathbb{R}^{2}\rightarrow\mathbb{R}$ satisfies all conditions listed in Theorem \ref{main theorem 1}, then $u$ is a representative of $b_{\alpha}$.
\end{The}

\begin{proof}
Choose any representative $u_{0}$ of $b_{\alpha}$. By Theorem \ref{gls}, the set of non-differentiable points of $u_{0}$ is exactly
$$
N_{\alpha}:=\{p\in\mathbb{R}^{2}|\text{ there exist more than one }\zeta\in\mathscr{R}_{\alpha}\text{ that emanate from $p$ }\}.
$$
Since $u_{0}$ is Lipschitz, $N_{\alpha}$ is a set of measure $0$. Denote by $N_{u}$ the non-differentiable points of $u$, by the second item of Theorem \ref{main theorem 1}, $N_{u}$ is also a set of measure $0$. By the definition of $N_{\alpha}$, there is a unique timelike ray $\zeta\in\mathscr{R}_{\alpha}$ emanating from any point of $\mathbb{R}^{2}\setminus N_{\alpha}$. Thus the third item of Theorem \ref{main theorem 1} implies $\nabla u=\nabla u_{0}$ on $\mathbb{R}^{2}\setminus N_{\alpha}\cup N_{u}$, which is a set of full measure. By Proposition \ref{wcf}, $u\equiv u_{0}+const.$ on $\mathbb{R}^{2}$.
\end{proof}

The case that $\alpha$ is rational is a bit more complicated but by no means difficult. 


\begin{The}\label{rational global vis}
For any  rational $\alpha\in(m^{-},m^{+})$, there is a continuous function on $\mathbb{R}^{2}$, denoted by $b_{\alpha}^{+}$ ($b_{\alpha}^{-}$), such that it satisfies all the requirements in Theorem \ref{main theorem 1}. Besides this, the set $\mathfrak{C}_{\alpha}$ associated to $b_{\alpha}^{+}$ ($b_{\alpha}^{-}$) which is defined in Theorem \ref{main theorem 1} exactly equals to $\mathscr{R}_{\alpha}^{per}\cup\mathscr{R}_{\alpha}^{+}$ ($\mathscr{R}_{\alpha}^{per}\cup\mathscr{R}_{\alpha}^{-}$).
\end{The}

\begin{proof}
For any timelike line $\gamma\in\mathscr{R}_{\alpha}^{per}$, choose a vector $h\in H_{1}(\mathbb{T}^{2},\mathbb{Z})\setminus(\mathfrak{T}\cup-\mathfrak{T})$ such that $T_{h}\gamma>\gamma$. Then we define a series of open domains $\Xi_{k}:=\{q\in\mathbb{R}^{2}|q<T_{kh}\gamma\}$. It is easy to see that $\Xi_{k}\subseteq\Xi_{k+1}$ and $\bigcup_{k=1}^{\infty}\Xi_{k}=\mathbb{R}^{2}$.

Let $b_{0}$ be a representative of $b_{\gamma}$ such that $b_{0}(p)=0$. Then it follows from Theorem \ref{gls} that $b_{0}$ and any representative of $b_{T_{h}\gamma}$ only differ by  a constant on $\Xi_{0}$. So we could choose a representative $b_{1}$ for $b_{T_{h}\gamma}$ such that $b_{1}(p)=0$. One succeed to define $b_{k},k\in\mathbb{N}$ as the representative for $b_{T_{kh}\gamma}$ and satisfies $b_{k}(p)=0$. Since the family of functions $\{b_{k}\}_{k\in\mathbb{N}}$ are uniformly Lipschitz w.r.t. $g_{R}$ and $b_{k}(p)=0$, it admits a limit, say a function $b_{\alpha}^{+}$, w.r.t. the  compact-open topology on $C^{0}(\mathbb{R}^{2},\mathbb{R})$.

Furthermore, for any positive integer $N$, it is obvious from the construction of $b_{k}$ that $b_{k},k\geq N$ are identical on $\Xi_{N}$, so $b_{\alpha}^{+}=b_{N}$ on $\Xi_{N}$. Thus $b_{\alpha}^{+}$ and $b_{N}$ have the same gradients on $\Xi_{N}$ and $b^+_{\alpha}$ is a Lipschitz viscosity solution to the eikonal equation $(*)$.  By the definition of $\Xi_{N}$, any timelike ray in $\mathscr{R}_{\alpha}$ that emanates from some point in $\Xi_{N}$ will stay in $\Xi_{N}$ forever. These observations show that $b_{\alpha}^{+}$ and $b_{N}$ have the same differentiability and the same gradient lines on $\Xi_{N}$. Since $\bigcup_{N=1}^{\infty}\Xi_{N}=\mathbb{R}^{2}$, it follows that $b_{\alpha}^{+}$ is defined on $\mathbb{R}^{2}$ and satisfies all items of Theorem \ref{main theorem 1}. The set of gradient lines for $b_{N}$ on $\Xi_{N}$ equals to
$$
\{\zeta\in\mathscr{R}_{\alpha}^{+}\cup\mathscr{R}_{\alpha}^{per}|\zeta<T_{Nh}\gamma\}.
$$
So we conclude that for $b_{\alpha}^{+}$, the set
$$
\mathfrak{C}_{\alpha}=\mathscr{R}_{\alpha}^{+}\cup\mathscr{R}_{\alpha}^{per}\subseteq\mathscr{R}_{\alpha},
$$
and that $b_{\alpha}^{+}$ differs by a constant if we choose another $\gamma$.

One could proceed the actions $T_{-kh},k\in\mathbb{N}$ on $\gamma$. By the same procedure as above, we could obtain another function $b_{\alpha}^{-}$, of which the set of gradient lines is
$$
\mathfrak{C}_{\alpha}=\mathscr{R}_{\alpha}^{-}\cup\mathscr{R}_{\alpha}^{per}\subseteq\mathscr{R}_{\alpha}.
$$

At last, if $\alpha$ is rational, the sets $\mathscr{R}_{\alpha}^{per}\cup\mathscr{R}_{\alpha}^{+}$ and $\mathscr{R}_{\alpha}^{per}\cup\mathscr{R}_{\alpha}^{-}$ are both $\mathbb{Z}^{2}$-invariant, thus by Proposition \ref{wcf} and Remark \ref{wcf'}, the differentials of $b_{\alpha}^{+},b_{\alpha}^{-}$ are both weakly closed 1-forms on $\mathbb{T}^{2}$.
\end{proof}

\begin{Lem}\label{6}
If $\alpha\in(m^{-},m^{+})$ is a rational asymptotic direction, $\gamma_{1}<\gamma_{2}<\gamma_{3}$ are periodic lines in $\mathscr{M}_{\alpha}^{per}$, then
\begin{equation}
\begin{split}
b_{\gamma_{1}}(\gamma_{3}(0))=b_{\gamma_{1}}(\gamma_{2}(0))+b_{\gamma_{2}}(\gamma_{3}(0)),\\
b_{\gamma_{3}}(\gamma_{1}(0))=b_{\gamma_{2}}(\gamma_{1}(0))+b_{\gamma_{3}}(\gamma_{2}(0)).
\end{split}
\end{equation}
Recall that $b_{\gamma_i}$ have been  normalized such that $b_{\gamma_i}(\gamma_i(0))=0$.
\end{Lem}

\begin{proof}
Denote by $\Pi_{ij}$ the closed strip  bounded by $\gamma_{i}$ and $\gamma_{j}$ (not necessary neighboring), where $i<j$ and $i,j=1,2,3$.
We shall only prove
\begin{equation}
b_{\gamma_{1}}(\gamma_{3}(0))=b_{\gamma_{1}}(\gamma_{2}(0))+b_{\gamma_{2}}(\gamma_{3}(0)),
\end{equation}
since the other equality can be proved in the same way.

By Theorem \ref{gls}, $db_{\gamma_{1}}=db_{\gamma_{2}}$ on $\Pi_{23}$, thus by the first item of Proposition \ref{wcf},
$$
b_{\gamma_{1}}(x)-b_{\gamma_{2}}(x)\equiv const.,\hspace{0.3cm}x\in\Pi_{23}.
$$
So we obtain
\begin{eqnarray*}
&&b_{\gamma_{1}}(\gamma_{3}(0))\\
&=&b_{\gamma_{1}}(\gamma_{2}(0))+(b_{\gamma_{1}}(\gamma_{3}(0))-b_{\gamma_{1}}(\gamma_{2}(0)))\\
&=&b_{\gamma_{1}}(\gamma_{2}(0))+(b_{\gamma_{2}}(\gamma_{3}(0))-b_{\gamma_{2}}(\gamma_{2}(0)))\\
&=&b_{\gamma_{1}}(\gamma_{2}(0))+b_{\gamma_{2}}(\gamma_{3}(0)).
\end{eqnarray*}
\end{proof}

\begin{The}\label{per fol}
Let $\underline{\gamma}<\overline{\gamma}$ be two distinct periodic timelike lines with the same rational asymptotic direction $\alpha$ and $b_{\underline{\gamma}}$ ($b_{\overline{\gamma}}$) be the Busemann function associated to $\underline{\gamma}$ ($\overline{\gamma}$) such that $b_{\underline{\gamma}}(\underline{\gamma}(0))=0$ ($b_{\overline{\gamma}}(\overline{\gamma}(0))=0$). Then the sum $b_{\underline{\gamma}}(\overline{\gamma}(0))+b_{\overline{\gamma}}(\underline{\gamma}(0))\geq0$, the equality holds if and only if the strip bounded by $\underline{\gamma}$ and $\overline{\gamma}$ admits a foliation with leaves in $\mathscr{M}^{per}_{\alpha}$.
\end{The}

\begin{Rem}
By Proposition \ref{subray bu} and Theorem \ref{gls}, the sum $b_{\underline{\gamma}}(\overline{\gamma}(0))+b_{\overline{\gamma}}(\underline{\gamma}(0))$ does not change if we choose another pair of initial points of $\underline{\gamma}$ and $\overline{\gamma}$ respectively.
\end{Rem}

\begin{proof}
Denote the minimal period of periodic timelike lines (with asymptotic direction $\alpha$) by $T$. Assume $(q,p)\in\overline{\alpha}$ and $p,q$ are relatively prime, $T$ is the Lorentzian length of maximal closed timelike curves on $(\mathbb{T}^{2},g)$ with homology class $(q,p)$. Since all maximal closed timelike curves with homology class $(q,p)$ have the same Lorentzian length, $T$ is independent of the choice of $\underline{\gamma},\overline{\gamma}$.

First, we assume that $\underline{\gamma}$ and $\overline{\gamma}$ are neighboring. By the definition of Busemann function,
\begin{equation}
\begin{split}
b_{\underline{\gamma}}(\overline{\gamma}(0))=\lim_{k\rightarrow\infty}[2kT-d(\overline{\gamma}(0),\underline{\gamma}(2kT)]\\
=\lim_{k\rightarrow\infty}[d(\underline{\gamma}(-kT),\underline{\gamma}(kT))-d(\overline{\gamma}(-kT),\underline{\gamma}(kT)].
\end{split}
\end{equation}
Interchanging $\underline{\gamma}$ with $\overline{\gamma}$ we get
\begin{equation}
b_{\overline{\gamma}}(\underline{\gamma}(0))=\lim_{k\rightarrow\infty}[d(\overline{\gamma}(-kT),\overline{\gamma}(kT))-d(\underline{\gamma}(-kT),\overline{\gamma}(kT)].
\end{equation}
Adding up the above two equations, we get
\begin{eqnarray*}
&&b_{\underline{\gamma}}(\overline{\gamma}(0))+b_{\overline{\gamma}}(\underline{\gamma}(0))\\
&=&\lim_{k\rightarrow\infty}[d(\underline{\gamma}(-kT),\underline{\gamma}(kT))+d(\overline{\gamma}(-kT),\overline{\gamma}(kT))\\
&-&d(\overline{\gamma}(-kT),\underline{\gamma}(kT)-d(\underline{\gamma}(-kT),\overline{\gamma}(kT)].
\end{eqnarray*}

On the other hand, by Theorem \ref{structure for lines}, there exist two timelike lines $\gamma^{-}\in\mathscr{M}_{\alpha}^{-}$ and $\gamma^{+}\in\mathscr{M}_{\alpha}^{+}$ such that $\underline{\gamma}<\gamma^{\pm}<\overline{\gamma}$. Thus for any sufficiently large $k$, there exist two points $\gamma^{-}(a_{k}),\gamma^{-}(b_{k})$ ($\gamma^{+}(c_{k}),\gamma^{+}(d_{k})$) on $\gamma^{-}$ ($\gamma^{+}$) such that:
\begin{equation}\label{eq:1}
\begin{split}
d_{R}(\overline{\gamma}(-kT),\gamma^{-}(a_{k}))<\frac{1}{4k};d_{R}(\underline{\gamma}(kT),\gamma^{-}(b_{k}))<\frac{1}{4k},\\
d_{R}(\underline{\gamma}(-kT),\gamma^{+}(c_{k}))<\frac{1}{4k};d_{R}(\overline{\gamma}(kT),\gamma^{+}(d_{k}))<\frac{1}{4k}.\\
\end{split}
\end{equation}
Here $a_{k},c_{k}\rightarrow-\infty,b_{k},d_{k}\rightarrow\infty$ as $k\rightarrow\infty$.

Let $L(\alpha)$ denote the Lipschitz constant of $d(\cdot,\cdot)$ w.r.t. $d_{R}(\cdot,\cdot)$. By Equation \ref{eq:1} and the fact that $\gamma^{-}$ intersects $\gamma^{+}$ transversally at a unique point $q_{0}=\gamma^{-}(s_{0})=\gamma^{+}(t_{0})$, we obtain that there exists an $\epsilon_{0}>0$ depending only on $\underline{\gamma}, \overline{\gamma}$ themselves such that for any sufficiently large $k$,
\begin{eqnarray*}
&&d(\underline{\gamma}(-kT),\underline{\gamma}(kT))+d(\overline{\gamma}(-kT),\overline{\gamma}(kT))+\frac{2L(\alpha)}{k}\\
&\geq&d(\gamma^{-}(a_{k}),\gamma^{+}(d_{k}))+d(\gamma^{+}(c_{k}),\gamma^{-}(b_{k}))+\frac{L(\alpha)}{k}\\
&\geq&L^{g}(\gamma^{+}|_{[a_{k},t_{0}]}\ast\gamma^{-}|_{[s_{0},d_{k}]})+L^{g}(\gamma^{+}|_{[c_{k},t_{0}]}\ast\gamma^{-}|_{[s_{0},b_{k}]})+2\epsilon_{0}+\frac{L(\alpha)}{k}\\
&=&L^{g}(\gamma^{+}|_{[a_{k},b_{k}]})+L^{g}(\gamma^{-}|_{[c_{k},d_{k}]})+2\epsilon_{0}+\frac{L(\alpha)}{k}\\
&\geq&d(\underline{\gamma}(-kT),\overline{\gamma}(kT))+d(\overline{\gamma}(-kT),\underline{\gamma}(kT))+2\epsilon_{0}.
\end{eqnarray*}
Denote $A(\underline{\gamma}, \overline{\gamma}):=b_{\underline{\gamma}}(\overline{\gamma}(0))+b_{\overline{\gamma}}(\underline{\gamma}(0))$ when $\underline{\gamma}, \overline{\gamma}$ are neighboring. The above inequality shows that
\begin{equation}\label{eq:2}
A(\underline{\gamma}, \overline{\gamma})\geq2\epsilon_{0}>0.
\end{equation}
By Lemma \ref{6}, we obtain
\begin{equation}\label{eq:3}
b_{\underline{\gamma}}(\overline{\gamma}(0))+b_{\overline{\gamma}}(\underline{\gamma}(0))=\sum_{i}A(\gamma_{i},\gamma_{i+1}),
\end{equation}
where the sum is taken over every pair of neighboring timelike lines $\gamma_{i}<\gamma_{i+1}$ in $\mathscr{M}_{\alpha}^{per}$ between $\underline{\gamma}$ and $\overline{\gamma}$.

By Equations \ref{eq:2} and \ref{eq:3}, the sum $b_{\underline{\gamma}}(\overline{\gamma}(0))+b_{\overline{\gamma}}(\underline{\gamma}(0))$ is zero if and only if there does not exist any pair of neighboring timelike periodic lines, this can only occur when the strip bounded by $\underline{\gamma}$ and $\overline{\gamma}$ admits a foliation by timelike periodic lines with asymptotic direction $\alpha$.
\end{proof}

\begin{The}\label{non-diff}
Let $\omega_{\alpha}^{+}$ and $\omega_{\alpha}^{-}$ denote the weakly closed 1-forms on $\mathbb{T}^{2}$ determined by the differentials of $b_{\alpha}^{+}$ and $b_{\alpha}^{-}$ respectively, then $[\omega_{\alpha}^{+}]=[\omega_{\alpha}^{-}]$ if and only if $\mathbb{R}^{2}$ admits a foliation whose leaves belong to $\mathscr{M}^{per}_{\alpha}$.
\end{The}

\begin{proof}
By the second item of Proposition \ref{wcf},
\begin{equation}
b_{\alpha}^{\pm}(x)=l_{\alpha}^{\pm}(x)+\sigma_{\alpha}^{\pm}(x),
\end{equation}
where $l_{\alpha}^{\pm}$ are linear on $\mathbb{R}^{2}$ and $\sigma_{\alpha}^{\pm}$ are periodic functions. Our aim is to prove that $l_{\alpha}^{+}\neq l_{\alpha}^{-}$.

By Theorem \ref{rational global vis}, for the Lorentzian Busemann function $b_{\gamma}$ associated to any timelike line $\gamma\in\mathscr{M}^{per}_{\alpha}$, there are two periodic functions $\sigma_{1},\sigma_{2}$ such that
\begin{equation}\label{eq:4}
\begin{split}
b_{\gamma}(x)=l_{\alpha}^{-}(x)+\sigma_{1}(x),\hspace{0.2cm}x>\gamma,\\
b_{\gamma}(x)=l_{\alpha}^{+}(x)+\sigma_{2}(x),\hspace{0.2cm}x<\gamma.
\end{split}
\end{equation}

Choosing a vector $h$ in $H_{1}(\mathbb{T}^{2},\mathbb{Z})\setminus(\mathfrak{T}\cup-\mathfrak{T})$ such that $T_{h}\gamma>\gamma$, by Equation \ref{eq:4} we obtain
\begin{equation}\label{eq:5}
\begin{split}
b_{\gamma}(T_{h}\gamma(0))-b_{\gamma}(\gamma(0))=l_{\alpha}^{-}(h),\\
b_{T_{h}\gamma}(T_{h}\gamma(0))-b_{T_{h}\gamma}(\gamma(0))=l_{\alpha}^{+}(h).
\end{split}
\end{equation}
Subtracting the second equation in Equations \ref{eq:5} from the first, we get
\begin{equation}\label{eq:6}
b_{\gamma}(T_{h}\gamma(0))+b_{T_{h}\gamma}(\gamma(0))=l_{\alpha}^{-}(h)-l_{\alpha}^{+}(h).
\end{equation}
By Theorem \ref{per fol}, $b_{\gamma}(T_{h}\gamma(0))+b_{T_{h}\gamma}(\gamma(0))$ is zero if and only if the strip bounded by $\gamma$ and $T_{h}\gamma$ admits a $\mathscr{M}_{\alpha}^{per}$-foliation, i.e. a foliation of which all leaves are in $\mathscr{M}_{\alpha}^{per}$. So by Equation \ref{eq:6}, $l_{\alpha}^{+}=l_{\alpha}^{-}$ if and only if $\mathbb{R}^{2}$ admits a foliation whose leaves belong to $\mathscr{M}_{\alpha}^{per}$.
\end{proof}

\section{Differentiability of the stable time separation}
S. Suhr established the existence of the stable time separation (which is the counterpart of the stable norm in the Riemannian case, see Definition \ref{std} in the Appedix) $\mathfrak{l}:\mathfrak{T}\rightarrow\mathbb{R}$ for general class A spacetimes $(M,g)$ in his paper \cite{Su2}. In this section, we shall discuss some the differentiability of $\mathfrak{l}$ when $M=\mathbb{T}^{2}$ and prove Theorem \ref{non-diff sts}.

At first, we shall give the definition and several basic properties of the stable time separation.

\begin{The}[{\cite[Theorem 4.1]{Su2}}]\label{sts}
Let $(M,g)$ be a class A spacetime. Then there exists a unique concave function $\mathfrak{l}:\mathfrak{T}\rightarrow\mathbb{R}$ such that for every $\epsilon>0$, there is a constant $C(\epsilon)<\infty$ with
\begin{enumerate}
  \item $|\mathfrak{l}(h)-d(x,y)|\leq C(\epsilon)$ for all $x,y\in\overline{M}$ with $y-x=h\in\mathfrak{T}^{\epsilon}$,
  \item $\mathfrak{l}(\lambda h)=\lambda\mathfrak{l}(h)$, for all $\lambda\geq0$,
  \item $\mathfrak{l}(h+h^{\prime})\geq\mathfrak{l}(h)+\mathfrak{l}(h^{\prime})$,
  \item $\mathfrak{l}(h)=\limsup_{h^{\prime}\rightarrow h}\mathfrak{l}(h)$ for $h\in\partial\mathfrak{T}$ and $h^{\prime}\in\mathfrak{T}$.
\end{enumerate}
We call $\mathfrak{l}$ the stable time separation.
\end{The}

The above theorem can be seen as a description for $\mathfrak{l}$ in the most general cases.
\begin{Rem}
For a general class A spacetime $(M,g)$, we note that:
\begin{enumerate}
  \item Since $\mathfrak{l}$ is concave on the stable time cone $\mathfrak{T}$, it is locally Lipschitz and differentiable almost everywhere on $\mathfrak{T}^{\circ}$.
  \item Since $\mathfrak{l}$ is linear on every half line $\overline{\alpha}\subseteq\mathfrak{T}$ emanating from the origin, it is differentiable along any radial direction on $\mathfrak{T}$.
\end{enumerate}
\end{Rem}

Let us return to our setting that $M=\mathbb{T}^{2}$ and $\overline{M}=\mathbb{R}^{2}$. Define
$$
D_{v}\mathfrak{l}(h):=\lim_{t\rightarrow0^{+}}\frac{\mathfrak{l}(h+tv)-\mathfrak{l}(h)}{t} ,$$
i.e. the direction derivative of $\mathfrak{l}$ at $h$ along $v$. By the concavity of $\mathfrak{l}$, $D_{v}\mathfrak{l}(h)$ always exists when $h\in\mathfrak{T}^{\circ}$. We notice that for a fixed $h\in\mathfrak{l}^{-1}(1)\cap\mathfrak{T}^{\circ}$, $D_{v}\mathfrak{l}(h)$ is positively homogenous of degree one as a function of direction $v$. We shall denote this function by $D_{\cdot}\mathfrak{l}(h)$. In general, we  say a concave function $f$ is differentiable at $x$ along straight line $L$ if $f$ is differentiable at $x$ when it is restricted to $L$.

Define the level set $\mathfrak{l}^{-1}(1)$ to be the unit sphere of $\mathfrak{l}$. By the second item in Theorem \ref{sts}, there is a unique $h\in\overline{\alpha}$ such that $\mathfrak{l}(h)=1$ for every $\alpha\in(m^{-},m^{+})$.

First we have the following lemma concerning the relationship between the direction derivatives of $\mathfrak{l}$ on its unit sphere and Busemann functions.

\begin{Lem}\label{5}
Let $(\mathbb{T}^{2},g)$ be a class A Lorentzian 2-torus and $\mathfrak{l}:\mathfrak{T}\rightarrow\mathbb{R}$ be the  stable time separation. Then for any $h\in\mathfrak{l}^{-1}(1)\cap\mathfrak{T}^{\circ}$ and any $\gamma\in\mathscr{M}_{\alpha}$ with $h\in \overline{\alpha}$, there is a constant $A$ depending only on $g,g_{R}$ and $\alpha$ such that $|b_{\gamma}(x)-D_{v}\mathfrak{l}(h)|\leq A$. Here, $v=x-\gamma(0)\in H_{1}(\mathbb{T}^{2},\mathbb{R})$.
\end{Lem}

\begin{proof}
For any fixed $x\in\mathbb{R}^{2}$, the mapping $\varphi_{x}$ which assigns every point $y\in\mathbb{R}^{2}$ to the vector $y-x\in H_{1}(\mathbb{T}^{2},\mathbb{R})$ is surjective.

For any $h\in\mathfrak{l}^{-1}(1)\cap\mathfrak{T}^{\circ}$, there is an asymptotic direction $\alpha\in(m^{-},m^{+})$ such that $h\in\overline{\alpha}\subseteq\mathfrak{T}^{\epsilon}$ for some $\epsilon>0$. By Theorem \ref{structure for lines}, every timelike line $\gamma$ with asymptotic direction $\alpha$ satisfies
\begin{equation}\label{ieq:12}
\|\gamma(T)-\gamma(0)-Th\|\leq \kappa(\epsilon,g,g_{R})
\end{equation}
for any $T\geq0$. Since $\mathfrak{l}$ is Lipschitz on $\mathfrak{T}^{\frac{\epsilon}{2}}$, we denote the Lipschitz constant by $L(\frac{\epsilon}{2})$. By Inequality \ref{ieq:12}, we have
\begin{equation}\label{ieq:8}
\begin{split}
&|\mathfrak{l}(Th+\gamma(0)-x)-\mathfrak{l}(\gamma(T)-x)|\\
\leq&L(\frac{\epsilon}{2})\|\gamma(T)-\gamma(0)-Th\|\\
\leq&L(\frac{\epsilon}{2})\kappa(\epsilon,g,g_{R})
\end{split}
\end{equation}
for any $T>0$.
By Theorem \ref{sts} and Proposition \ref{asym'}, there exist constants $T_0>0 $ and $C(\frac{\epsilon}{2})$ such that
\begin{equation}\label{ieq:9}
|\mathfrak{l}(\gamma(T)-x)-d(x,\gamma(T))|\leq C(\frac{\epsilon}{2})
\end{equation}
for $T\geq T_{0}>0$. Thus, we obtain
\begin{equation}
\begin{split}
&|T-d(x,\gamma(T))+T[\mathfrak{l}(h+\frac{\gamma(0)-x}{T})-\mathfrak{l}(h)]|\\
=&|T\mathfrak{l}(h+\frac{\gamma(0)-x}{T})-d(x,\gamma(T))|\\
=&|\mathfrak{l}(Th+\gamma(0)-x)-d(x,\gamma(T))|\\
\leq&|\mathfrak{l}(Th+\gamma(0)-x)-\mathfrak{l}(\gamma(T)-x)|+|\mathfrak{l}(\gamma(T)-x)-d(x,\gamma(T))|\\
\leq&L(\frac{\epsilon}{2})\kappa(\epsilon,g,g_{R})+C(\frac{\epsilon}{2}).
\end{split}
\end{equation}
Here, the first and second equalities follow from the fact that $h\in\mathfrak{l}^{-1}(1)$ and the second item of Theorem \ref{sts}, the first inequality follows by adding up the Inequalities \ref{ieq:8}, \ref{ieq:9}.
Setting  $v=x-\gamma(0)$ and letting $T$ go to infinity, we obtain
\begin{equation}
|b_{\gamma}(x)-D_{v}\mathfrak{l}(h)|\leq L(\frac{\epsilon}{2})\kappa(\epsilon,g,g_{R})+C(\frac{\epsilon}{2}).
\end{equation}
By choosing $A=L(\frac{\epsilon}{2})\kappa(\epsilon,g,g_{R})+C(\frac{\epsilon}{2})$, we complete the proof.
\end{proof}

The second lemma concerns the equivalence of the differentiability of the stable time separation and its unit sphere.
\begin{Lem}\label{dus}
Let $(\mathbb{T}^{2},g)$ be a class A Lorentzian 2-torus and $\mathfrak{l}:\mathfrak{T}\rightarrow\mathbb{R}$ its stable time separation. For  $h\in\mathfrak{T}^{\circ}$, the unit sphere of $\mathfrak{l}$ is differentiable at $h$ if and only if $\mathfrak{l}$ is differentiable at $h$.
\end{Lem}

\begin{proof}
Fix a point $h$ on $\mathfrak{l}^{-1}(1)\cap\mathfrak{T}^{\circ}$, then $h\in\overline{\alpha}$ for some $\alpha\in(m^{-},m^{+})$. From Theorem \ref{sts}, $\mathfrak{l}$ is concave on $\mathfrak{T}^{\circ}$, then the set $C(1):=\{h\in\mathfrak{T}^{\circ}|\mathfrak{l}(h)\geq1\}$ is convex. So the unit sphere of $\mathfrak{l}$ is the boundary of $C(1)$ and is locally a graph of a convex function.

Assume that $\mathfrak{l}$ is differentiable at $h$, denote the differential of $\mathfrak{l}$ at $h$ by $P$. If $\mathfrak{l}^{-1}(1)$ is non-differentiable at $h$, there are two linearly independent vectors $h_{1},h_{2}\in T_{h}\mathbb{R}^{2}$ (this follows from $\mathfrak{l}^{-1}(1)$ is locally a graph of a convex function) such that $\langle P,h_{1}\rangle=\langle P,h_{2}\rangle=0$. So we get $P=0$ which contradicts the fact that $\mathfrak{l}$ grows linearly on $\overline{\alpha}$.

Assume that $\mathfrak{l}^{-1}(1)$ is differentiable at $h$, then there exists a nonzero vector $h^{\prime}$ such that the curve $h+th^{\prime},t\in(-\delta,\delta)$ is tangent to $\mathfrak{l}^{-1}(1)$ at $h$ and $h,h^{\prime}$ are linearly independent. In other words, $\mathfrak{l}$ is differentiable at $h$ along two staight lines with directional vectors $h$ (by linearity) and $h^{\prime}$. By Lemma \ref{convex diff}, we conclude that $\mathfrak{l}$ is differentiable at $h$.
\end{proof}

Now we can prove Theorem \ref{non-diff sts}. We will call $h\in\mathfrak{l}^{-1}(1)$ rational (irrational) if $h\in\overline{\alpha}$ with some rational (irrational) asymptotic direction $\alpha$.

\noindent \textit{Proof of Theorem \ref{non-diff sts}.} By Lemma \ref{dus}, it is sufficient to prove $\mathfrak{l}:\mathfrak{T}^{\circ}\rightarrow\mathbb{R}$ is differentiable (non-differentiable) at any irrational (rational) $h\in\mathfrak{l}^{-1}(1)$. By Lemma \ref{convex diff}, this amounts to prove that there exists a non-radial direction $v$ such that $\mathfrak{l}|_{\{h+tv|t\in\mathbb{R}\}\cap\mathfrak{T}^{\circ}}$ is differentiable (or non-differentiable) at an irrational (a rational) $h$. Since $\mathfrak{l}$ is concave on $\mathfrak{T}^{\circ}$, we only need to prove there exists a non-radial direction $v$ such that $D_{v}\mathfrak{l}(h)=(\neq)-D_{-v}\mathfrak{l}(h)$ when $h$ is irrational (rational). In the following, let $b_{\gamma}$ be the Lorentzian Busemann function for some timelike line $\gamma$ with asymptotic direction in $(m^{-},m^{+})$.

Assume that $\alpha\in(m^{-},m^{+})$ is irrational and $\gamma\in\mathscr{M}_{\alpha}$. By Theorem \ref{irrational global vis}, we could write $b_{\gamma}(x)=l_{\alpha}(x)+\sigma_{\gamma}(x)$, where $l_{\alpha}$ is  a linear function  depending only on $\alpha$, $\sigma_{\gamma}$ is a periodic function on $\mathbb{R}^{2}$. Choose any nonzero $v\in H_{1}(\mathbb{T}^{2},\mathbb{Z})$, then $v\notin\overline{\alpha}$. By Lemma \ref{5}, for any $k\in\mathbb{Z}$, there exists some constant $A$ such that
\begin{equation}
|l_{\alpha}(kv)-D_{kv}\mathfrak{l}(h)|\leq A.
\end{equation}
By the linearity of $l_{\alpha}(\cdot)$ and the positive homogeneity of $D_{\cdot}\mathfrak{l}(h)$,
\begin{equation}
D_{v}\mathfrak{l}(h)=l_{\alpha}(v)=-l_{\alpha}(-v)=D_{-v}\mathfrak{l}(h).
\end{equation}

Assume that $\alpha\in(m^{-},m^{+})$ is rational and $\gamma\in\mathscr{M}_{\alpha}^{per}$. By Theorem \ref{gls} and Theorem \ref{rational global vis}, we can write
\begin{equation}
\begin{split}
b_{\gamma}(x)=l_{\alpha}^{-}(x)+\sigma_{1}(x),\hspace{0.2cm}x>\gamma,\\
b_{\gamma}(x)=l_{\alpha}^{+}(x)+\sigma_{2}(x),\hspace{0.2cm}x<\gamma,
\end{split}
\end{equation}
where $l_{\alpha}^{\pm}$ are linear functions  depending only on $\alpha$, $\sigma_{1,2}$ are periodic functions on $\mathbb{R}^{2}$. We choose a vector $v$ in $H_{1}(\mathbb{T}^{2},\mathbb{Z})\setminus(\mathfrak{T}\cup-\mathfrak{T})$ such that $T_{v}\gamma>\gamma$. By Lemma \ref{5}, for any $k\in\mathbb{N}$, there exists some constant $A$ such that
\begin{equation}
\begin{split}
|l_{\alpha}^{-}(kv)-D_{kv}\mathfrak{l}(h)|\leq A,\\
|l_{\alpha}^{+}(-kv)-D_{-kv}\mathfrak{l}(h)|\leq A.
\end{split}
\end{equation}
By Theorem \ref{non-diff} and the positive homogeneity of $l_{\alpha}^{\pm}(\cdot)$ and $D_{\cdot}\mathfrak{l}(h)$,
\begin{equation}
D_{v}\mathfrak{l}(h)=l_{\alpha}^{-}(v)\neq-l_{\alpha}^{+}(-v)=D_{-v}\mathfrak{l}(h).
\end{equation}\qed

\begin{Rem}
In the setting of class A Lorentzian 2-tori, by almost the same discussion as in \cite[Section 4, Page 386]{Ma}, it is easy to prove that $\mathfrak{l}^{-1}(1)$ is strictly concave on $\mathfrak{T}^{\circ}$.
\end{Rem}

\section{Appendix}
In this section, we collect part of elementary concepts and results in global Lorentzian geometry which are frequently used in this article. For a comprehensive introduction to this topic, see standard textbooks \cite{B-E-E}, \cite{O'N-B book}. Our presentation strongly relies on Suhr's papers \cite{Su1}, \cite{Su2}.

First, we shall state a result concerning a closed Riemannian manifold $(M,g_{R})$.
\begin{The}[\cite{Bu}]\label{std}
Let $(M,g_{R})$ be a compact Riemannian manifold. Then there exists a unique norm $\|\cdot\|:H_{1}(M,\mathbb{R})\rightarrow\mathbb{R}$ and a constant std$(g_{R})<\infty$ such that
$$
|\text{dist}(x,y)-\|x-y\||\leq\text{std}(g_{R})
$$
for any $x,y\in\overline{M}$.
Here $\|\cdot\|$ is called to be  the stable norm of $g_{R}$ on $H_{1}(M,\mathbb{R})$.
\end{The}

In Riemannian geometry, there are some important concepts that relate the Riemannian structure to the metric or topology structure. The same concepts also lie at the foundation of Loretzian geometry. But on the contrary, they are far from being well-known. Only several general properties are proved.

\begin{defn}\label{length}
Let $(M,g)$ be a spacetime and $\gamma:[a,b]\rightarrow M$ a causal curve, the Lorentzian length of $\gamma$ is defined by
$$
L^{g}(\gamma):=\int_{a}^{b}\sqrt{-g(\dot{\gamma}(t),\dot{\gamma}(t))}dt.
$$
Since causal curves always admit Lipschitz parametrization, the above definition makes no confusion.
\end{defn}

\begin{Pro}\label{usc}
If a sequence of causal curves $\gamma_{n}:[a,b]\rightarrow M$, parameterized by the arclength w.r.t. $g_{R}$, converges uniformly to a causal curve $\gamma:[a,b]\rightarrow M$, then
$$
L^{g}(\gamma)\geq\limsup_{n\rightarrow\infty}L^{g}(\gamma_{n}).
$$
\end{Pro}

\begin{defn}\label{dis}
Let $(M,g)$ be a spacetime, we define the time separation or the Lorentzian distance function as $d(p,q):=sup\{L^{g}(\gamma):\gamma\in C^{+}(p,q)\}$, where $C^{+}(p,q)$
denotes the set of future-directed causal curves connecting $p$ with $q$. If $C^{+}(p,q)=\emptyset$, then set $d(p,q):=0$.
\end{defn}

\begin{Pro}
For general spacetime $(M,g)$, the time separation is only lower semicontinuous on $M\times M$. If $(M,g)$ is globally hyperbolic, the time separation is continuous and there exists a maximal causal geodesic connecting $p$ with $q$ for all $q\in J^{+}(p)$.
\end{Pro}

The following definition gives a picture of what the time looks like in the theory of general relativity.
\begin{defn}\label{fun}
Let $(M,g)$ be a spacetime. A function $\tau:M\rightarrow\mathbb{R}$ is called
\begin{enumerate}
  \item a time function if it is continuous and strictly increasing on each future-directed causal curve in $(M,g)$;
  \item a temporal function if it is $C^{1}$ and has a past-directed timelike gradient at every point on $M$.
\end{enumerate}
\end{defn}

\begin{Rem}
Any temporal function on $(M,g)$ must be a time function, and if it satisfies the Lorentzian eikonal equation $g(\nabla\tau,\nabla\tau)=-1$, it is then a global viscosity (in fact, classical) solution to the Lorentzian eikonal equation. The latter case could happen only when $M$ is noncompact.
\end{Rem}

Finally, we recall two important properties of the so called class $A_{1}$ spacetimes defined in \cite[Sections 2 and 3]{Su3}. As S. Suhr has proved in \cite{Su3}, class A Lorentzian 2-tori are also class $A_{1}$, so these two properties apply to our case. They are crucial for the proof of our main results.

The first concerns how the global behavior of maximal curves effects their local property.
\begin{Pro}[{\cite[Proposition 4.1]{Su4}}]\label{A1-1}
Let $(M,g)$ be of class $A_{1}$. Then for any $\epsilon>0$ there exist $\delta>0$ and $K<\infty$ such that
$$
\dot{\gamma}(t)\in\text{Time}^{\epsilon}(M^{2},[g])_{\gamma(t)}
$$
for all maximizers $\gamma:[a,b]\rightarrow M$ with $\gamma(b)-\gamma(a)\in\mathfrak{T}^{\delta}\setminus B_{K}(0)$ and all $t\in[a,b]$.
\end{Pro}

The second contains an answer of a problem which lies at the foundation of the global Lorentzian geometry. The problem is to find some appropriate spacetime on which the time separation (or Lorentzian distance function) is Lipschitz w.r.t some Riemannian structure.
\begin{Pro}[{\cite[Theorem 4.3]{Su4}}]\label{A1-2}
Let $(M,g)$ be of class $A_{1}$. Then for any $\epsilon>0$ there exist constants $K(\epsilon),L(\epsilon)<\infty$ such that $(x,y)\mapsto d(x,y)$ is $L(\epsilon)$-Lipschitz on $\{(x,y)\in\overline{M}\times\overline{M}|y-x\in\mathfrak{T}^{\epsilon}\setminus B_{K(\epsilon)}(0)\}$.
\end{Pro}

A basic lemma in convex analysis is needed in the proof of Lemma \ref{dus}. For completeness, we give the  proof here.
\begin{Lem}\label{convex diff}
Let $f:\mathbb{R}^{n}\rightarrow\mathbb{R}$ be a concave function. Suppose $f$ is differentiable at $x$ along straight lines $L_{i}$, where $L_{i}:=\{x+tV_{i}|t\in\mathbb{R}\}$, $V_{1},V_{2},...,V_{n}$ are $n$ linearly independent  vectors, then $f$ is differentiable at $x$.
\end{Lem}

\begin{proof}
Fix a point $x\in\mathbb{R}^{n}$, denote the set of super-gradients of $f$ at $x$ by $D^{+}f(x)$ (namely, $D^{+}f(x)=\{p\in T_{x}\mathbb{R}^{n}|f(y)-f(x)-\langle p,y-x\rangle\leq0,\text{for any }y\in\mathbb{R}^{n}\}$) and the directional derivative of $f$ at $x$ along $V$ by $D_{V}f(x)$. Since $f$ is differentiable along straight lines $L_{i}$, by using the concavity of $f$, we have
\begin{equation}
-\langle P,-V_{i}\rangle\leq-D_{-V_{i}}f(x)=D_{V_{i}}f(x)\leq\langle P,V_{i}\rangle
\end{equation}
for any $P\in D^{+}f(x)$.
So $\langle P,V_{i}\rangle=D_{V_{i}}f(x)$ for any $P\in D^{+}f(x)$. Now let $P,P^{\prime}$ be two elements in $D^{+}f(x)$, then we have
\begin{equation}
\langle P-P^{\prime},V_{i}\rangle=0.
\end{equation}
Since $V_{i}(i=1,...n)$ are linearly independent, we must have $P-P^{\prime}=0$. Thus the set of super-gradients of $f$ at $x$ degenerates to a singleton and $f$ is differentiable at $x$.
\end{proof}

\textbf{Acknowledgements}
The first author would like to thank Professor C.-Q. Cheng for leading him to the topic of Aubry-Mather theory for relativistic mechanical system and for sharing much research experience with him. The second author would like to thank Professor V. Bangert for providing a paper copy of E. Schelling's Diplomarbeit \cite{Sc}. Both authors would like to thank Dr. S. Suhr for giving a mini-course on Lorentzian Aubry-Mather theory at Nanjing University on January, 2015, which leads them into this beautiful field. Finally, both authors want to thank Professor C.-Q. Cheng and Professor W. Cheng for many helps and encouragements.

\end{document}